\documentclass[english]{article}
\usepackage[T1]{fontenc}
\usepackage[latin1]{inputenc}
\usepackage{geometry}
\usepackage{float}
\usepackage{mathrsfs}
\usepackage{amsmath}
\usepackage{amssymb}
\usepackage{graphicx}
\usepackage{subfigure}

\makeatletter

\floatstyle{ruled}
\newfloat{algorithm}{tbp}{loa}
\providecommand{\algorithmname}{Algorithm}
\floatname{algorithm}{\protect\algorithmname}

\usepackage{amsthm}

\usepackage{mathrsfs}

\usepackage{amsfonts}

\usepackage{epsfig}

\usepackage{bm}

\usepackage{mathrsfs}

\usepackage{enumerate}

\@ifundefined{definecolor}{\@ifundefined{definecolor}
 {\@ifundefined{definecolor}
 {\usepackage{color}}{}
}{}
}{}

\usepackage{subfig}\usepackage[all]{xy}

\newtheorem{lem}{Lemma}[section]
\newtheorem{rem}{Remark}[section]
\newtheorem{prop}{Proposition}[section]

\newcounter{hypA}

\newcounter{hypB}

\newcounter{hypD}
\newenvironment{hypD}{\refstepcounter{hypD}\begin{itemize}
 \item[({\bf D\arabic{hypD}})]}{\end{itemize}}

\usepackage{babel}\date{}

\usepackage{babel}

\makeatother

\usepackage{babel}

\begin{document}

\begin{center}

{\Large \textbf{Unbiased Estimation of the Solution to Zakai's Equation}}

\vspace{0.5cm}

BY HAMZA M. RUZAYQAT \& AJAY JASRA

{\footnotesize Computer, Electrical and Mathematical Sciences and Engineering Division, King Abdullah University of Science and Technology, Thuwal, 23955, KSA.}
{\footnotesize E-Mail:\,} \texttt{\emph{\footnotesize hamza.ruzayqat@kaust.edu.sa, ajay.jasra@kaust.edu.sa}}
\end{center}

\begin{abstract}
In the following article we consider the non-linear filtering problem in continuous-time
and in particular the solution to Zakai's equation or the normalizing constant.
We develop a methodology to produce finite variance, almost surely
unbiased estimators of the solution to Zakai's equation. That is, given
access to only a first order discretization of solution to the Zakai equation, we present a method
which can remove this discretization bias. The approach, under assumptions, is proved to have finite variance and
is numerically compared to using a particular multilevel Monte Carlo method.
\end{abstract}

\section{Introduction}

Let $(\Omega,\mathcal{F})$ be a measurable space. On $(\Omega,\mathcal{F})$
consider the probability measure $\mathbb{P}$ and a pair of stochastic processes $\{Y_t\}_{t\geq 0}$, $\{X_t\}_{t\geq 0}$, with $Y_t\in\mathbb{R}^{d_y}$, $X_t\in\mathbb{R}^{d_x}$ $(d_{y},d_x)\in\mathbb{N}^2$, $d_{x},d_y<+\infty$, with $X_0=x_*\in\mathbb{R}^{d_x}$ given:
\begin{eqnarray}
dY_t & = & h(X_t)dt + dB_t \label{eq:obs}\\
dX_t & = & b(X_t)dt + \sigma(X_t)dW_t \label{eq:state}
\end{eqnarray}
where $h:\mathbb{R}^{d_x}\rightarrow\mathbb{R}^{d_y}$, $b:\mathbb{R}^{d_x}\rightarrow\mathbb{R}^{d_x}$, $\sigma:\mathbb{R}^{d_x}\rightarrow\mathbb{R}^{d_x\times d_x}$ with $\sigma$ non-constant and of full rank and $\{B_t\}_{t\geq 0}, \{W_t\}_{t\geq 0}$ 
are independent standard Brownian motions of dimension $d_y$ and $d_x$ respectively. Let $\{\mathcal{F}_t\}$ be a filtration on $\mathcal{F}$ such that $\{B_t\}_{t\geq 0}$ and 
$\{W_t\}_{t\geq 0}$
are independent standard Brownian motions. Let $T>0$ be an arbitrary real number 
and introduce the probability measure $\overline{\mathbb{P}}$ which is equivalent to $\mathbb{P}$ on $\mathcal{F}_T$ defined by the Radon-Nikodym derivative 
$$
Z_{T}:=\frac{d\mathbb{P}}{d\overline{\mathbb{P}}} = \exp\Big\{\int_{0}^T h(X_s)^*dY_s - \frac{1}{2}\int_{0}^Th(X_s)^*h(X_s)ds\Big\}
$$
with, under $\overline{\mathbb{P}}$, $\{X_t\}_{t\geq 0}$ following the dynamics \eqref{eq:state} and independently $\{Y_t\}_{t\geq 0}$ is a standard Brownian motion. We have the solution to the Zakai equation (e.g.~\cite[Theorem 3.2.4.]{crisan_bain}) for $\varphi\in\mathcal{B}_b(\mathbb{R}^{d_x})$ (bounded and measurable real valued functions)
$$
\gamma_{t}(\varphi) := \overline{\mathbb{E}}\Big[\varphi(X_t)\exp\Big\{\int_{0}^t h(X_s)^*dY_s - \frac{1}{2}\int_{0}^th(X_s)^*h(X_s)ds\Big\}\Big|\mathcal{Y}_t\Big]
$$
where $\mathcal{Y}_t$ is the filtration generated by the process $\{Y_s\}_{0\leq s \leq t}$. Our objective is to, recursively in time, estimate $\gamma_{t}(\varphi)$ over some
finite and regular time grid. The solution of Zakai's equation can be useful for model selection in statistics or as a solution of a particular stochastic partial differential equation in applied mathematics.

In most cases of practical interest, one only has access to finite time discretization of the data and so one must often, correspondingly, discretize the functionals associated to the signal and observations. Several possibilities have been considered in the literature (e.g.~\cite{crisan_ortiz1,crisan_ortiz2}) but we use the first order approach in \cite{picard}. Even given this
discretization, one must often time-discretize \eqref{eq:state} of which we use the Euler method. Once one has reached this stage, the problem of numerically approximating
Zakai's equation corresponds to that of approximating the normalizing constant of a high-frequency state-space model, of which there is now a rather mature collection of methods
for doing so, for instance, based upon particle filters (PF).

Given a state-space model for which one can sample from the hidden Markov chain and evaluate the conditional likelihood of an observation given the state, a particle
filter provides consistent Monte Carlo estimates of the filter and unbiased estimates of the marginal likelihood; see for instance \cite{delmoral}. In the context of the model \eqref{eq:obs}-
\eqref{eq:state}, after discretization, many particle filter approaches have been suggested in the literature \cite{crisan_bain,delmoral,fearn}. We follow the methods considered in \cite{high_freq_ml}, who
apply multilevel particle filters for the approximation of the filtering problem. The multilevel Monte Carlo (MLMC) method e.g.~\cite{giles,giles1,hein} is often used for problems where one is interested in the estimation
of an expectation w.r.t.~a probability law that has been discretized, for instance the law of a diffusion at some given time $T$, which has been Euler discretized. The idea
is to present a telescoping sum representation of the expectation under a given precise discretization, in terms of differences of expectations of increasingly coarse discretizations.
If one can sample from appropriate couplings of the probability laws in the differences, then one can reduce the computational effort to achieve a pre-specified mean square error, versus simply considering approximating the expectation associated to the precise discretization by itself. Detailed reviews of these methods can be found in \cite{giles1,ml_rev}.

The methodology of this article concerns the estimation of the solution of Zakai's equation and in particular an estimate that is almost surely unbiased, in that the discretization error is removed from the estimate. The approach that we
present is based upon the unbiased methods of \cite{rhee} (see also \cite{mcl,matti}) combined with the multilevel methodology in \cite{high_freq_ml}. More precisely, we start by presenting a multilevel identity for the approximation of the solution of Zakai's equation which is biased, in terms of the discretization error. We prove that, for a particular implementation based on the algorithm in \cite{high_freq_ml},  in order to obtain a mean square
error of $\mathcal{O}(\epsilon^2)$, $\epsilon>0$, the computational effort required is $\mathcal{O}(\epsilon^{-3})$, versus $\mathcal{O}(\epsilon^{-4})$ if one does not use
a multilevel strategy. Then, given access to high-frequency data, we show how this identity can be randomized
to remove the discretization error of the multilevel method. We prove that our proposed estimators are unbiased and of finite variance. We also consider the computational effort
to produce our estimate relative to using multilevel approaches, both theoretically and numerically. In particular, we demonstrate that to be within $\epsilon>0$ of the true solution to
Zakai's equation (with high probability) one requires a computational effort of $\mathcal{O}(\epsilon^{-3}|\log_2(\epsilon)|^{3+\beta})$, for some $\beta>0$. Thus there is an
extra cost to pay for unbiasedness, relative to using multilevel methods. We remark however, that the unbiased method is exceptionally amenable to parallel implementation, especially relative to using afore-mentioned multilevel approach, and this element is not considered in our mathematical analysis. In addition, our unbiased methodology provides
a ground truth estimate, which can be useful if one resorts to estimation methods which exhibit discretization bias.

This article is structured as follows. In Section \ref{sec:review} we provide a review of the methodology to be used. In Section \ref{sec:method} our method is presented.
In Section \ref{sec:theory} we show that our method produces unbiased and finite variance estimators. We also present a result associated to a multilevel estimator.
In Section \ref{sec:numerics} our numerical results are presented. The appendix features technical results for the proofs of our theoretical results.

\section{Review of Relevant Methodology}\label{sec:review}

The following Section will provide a review of the methodology to be used in this article. The section is structured as follows. 
We first describe our notation in Section \ref{sec:notat}.
In Section 
\ref{sec:disc_model} we describe the discretized model that is the one that we will work with in practice. In Section \ref{sec:mlmc_rev}
we review the multilevel Monte Carlo method, which will be used in this article and is an approach which can reduce the cost of estimation, relative
to Monte Carlo, to achieve a given mean square error (MSE), particularly in problems which are subject to discretization. 
The next two Sections \ref{sec:pf_rev} and \ref{sec:cpf_rev} review methodology which can be used to implement MLMC for the class of problems considered in this article.
The final Section \ref{sec:mlpf_rev} summarizes MLMC implemented via particle and coupled particle filters (CPF); the multilevel particle filter (MLPF). Throughout the article we assume that all the random variables that are mentioned are well-defined on the measurable space $(\Omega,\mathcal{F})$.

\subsection{Notations}\label{sec:notat}

Let $(\mathsf{X},\mathcal{X})$ be a measurable space.
For $\varphi:\mathsf{X}\rightarrow\mathbb{R}$ we write $\mathcal{B}_b(\mathsf{X})$ as the collection of bounded measurable functions. 
Let $\varphi:\mathbb{R}^d\rightarrow\mathbb{R}$, $\textrm{Lip}_{\|\cdot\|_2}(\mathbb{R}^{d})$ denotes the collection of real-valued functions that are Lipschitz w.r.t.~$\|\cdot\|_2$ ($\|\cdot\|_p$ denotes the $\mathbb{L}_p-$norm of a vector $x\in\mathbb{R}^d$). That is, $\varphi\in\textrm{Lip}_{\|\cdot\|_2}(\mathbb{R}^{d})$ if there exists a $C<+\infty$ such that for any $(x,y)\in\mathbb{R}^{2d}$
$
|\varphi(x)-\varphi(y)| \leq C\|x-y\|_2.
$
We write $\|\varphi\|_{\textrm{Lip}}$ as the Lipschitz constant of a function $\varphi\in\textrm{Lip}_{\|\cdot\|_2}(\mathbb{R}^{d})$.
For $\varphi\in\mathcal{B}_b(\mathsf{X})$, we write the supremum norm $\|\varphi\|=\sup_{x\in\mathsf{X}}|\varphi(x)|$.
$\mathcal{P}(\mathsf{X})$  denotes the collection of probability measures on $(\mathsf{X},\mathcal{X})$.
For a measure $\mu$ on $(\mathsf{X},\mathcal{X})$
and a $\varphi\in\mathcal{B}_b(\mathsf{X})$, the notation $\mu(\varphi)=\int_{\mathsf{X}}\varphi(x)\mu(dx)$ is used. 
$\mathsf{B}(\mathbb{R}^d)$ denote the Borel sets on $\mathbb{R}^d$. $dx$ is used to denote the Lebesgue measure.
Let $K:\mathsf{X}\times\mathcal{X}\rightarrow[0,\infty)$ be a non-negative operator and $\mu$ be a measure then we use the notations
$
\mu K(dy) = \int_{\mathsf{X}}\mu(dx) K(x,dy)
$
and for $\varphi\in\mathcal{B}_b(\mathsf{X})$, 
$
K(\varphi)(x) = \int_{\mathsf{X}} \varphi(y) K(x,dy).
$
For $A\in\mathcal{X}$ the indicator is written $\mathbb{I}_A(x)$.
$\mathcal{N}_s(\mu,\Sigma)$ (resp.~$\psi_s(x;\mu,\Sigma)$)
denotes an $s-$dimensional Gaussian distribution (density evaluated at $x\in\mathbb{R}^s$) of mean $\mu$ and covariance $\Sigma$. If $s=1$ we omit the subscript $s$. For a vector/matrix $X$, $X^*$ is used to denote the transpose of $X$.
For $A\in\mathcal{X}$, $\delta_A(du)$ denotes the Dirac measure of $A$, and if $A=\{x\}$ with $x\in \mathsf{X}$, we write $\delta_x(du)$. 
$\mathcal{U}_A$ is used to denote the uniform distribution on a set $A$.
For a vector-valued function in $d-$dimensions (resp.~$d-$dimensional vector), $\varphi(x)$ (resp.~$x$) say, we write the $i^{\textrm{th}}-$component ($i\in\{1,\dots,d\}$) as $\varphi^{(i)}(x)$ (resp.~$x^{(i)}$). For a $d\times q$ matrix $x$ we write the $(i,j)^{\textrm{th}}-$entry as $x^{(ij)}$.
For $\mu\in\mathcal{P}(\mathsf{X})$ and $X$ a random variable on $\mathsf{X}$ with distribution associated to $\mu$ we use the notation $X\sim\mu(\cdot)$. 

\subsection{Discretized Model}\label{sec:disc_model}

The following section is taken from \cite{high_freq_ml}.
To minimize certain technical difficulties, the following assumption is made throughout the paper:
\begin{hypD}\label{hyp_diff:1}
We have:
\begin{enumerate}
\item{$\sigma^{(ij)}$ is bounded with $\sigma^{(ij)}\in\textrm{Lip}_{\|\cdot\|_2}(\mathbb{R}^{d_x})$, $(i,j)\in\{1,\dots,d_x\}^2$ and $a(x):=\sigma(x)\sigma(x)^*$ is uniformly elliptic.}
\item{$(h^{(i)},b^{(j)})$ are bounded and $(h^{(i)},b^{(j)})\in\textrm{Lip}_{\|\cdot\|_2}(\mathbb{R}^{d_x})\times\textrm{Lip}_{\|\cdot\|_2}(\mathbb{R}^{d_x})$, $(i,j)\in\{1,\dots,d_y\}\times\{1,\dots,d_x\}$.}
\end{enumerate}
\end{hypD}

In practice, we will have to work with a discretization of the model in \eqref{eq:obs}-\eqref{eq:state}. We will assume access to path of data $\{Y_t\}_{0\leq t \leq T}$ which is observed at a high frequency. 

Let $l\in\{0,1,\dots,\}$ be given and consider an Euler discretization of step-size $\Delta_l=2^{-l}$, $k\in\{1,2,\dots,2^lT\}$, $\widetilde{X}_{0}=x_*$:
\begin{eqnarray}
\widetilde{X}_{k\Delta_l} & = & \widetilde{X}_{(k-1)\Delta_l} + b(\widetilde{X}_{(k-1)\Delta_l})\Delta_l + \sigma(\widetilde{X}_{(k-1)\Delta_l})[W_{k\Delta_l}-W_{(k-1)\Delta_l}].\label{eq:disc_state}
\end{eqnarray}
It should be noted that the Brownian motion in \eqref{eq:disc_state} is the same as in \eqref{eq:state} under both $\mathbb{P}$ and $\overline{\mathbb{P}}$.
Then, for $k\in\{0,1,\dots\}$ define:
$$
G_{k}^l(x_{k\Delta_l}) := \exp\Big\{h(x_{k\Delta_l})^*(y_{(k+1)\Delta_l}-y_{k\Delta_l})-\frac{\Delta_l}{2}h(x_{k\Delta_l})^*h(x_{k\Delta_l})\Big\}
$$
and note that for any $T\in\mathbb{N}$
$$
Z_{T}^l(x_0,x_{\Delta_l},\dots,x_{T-\Delta_l}) := \prod_{k=0}^{2^lT-1}G_{k}^l(x_{k\Delta_l}) = \exp\Big\{\sum_{k=0}^{2^lT-1}\Big[h(x_{k\Delta_l})^*(y_{(k+1)\Delta_l}-y_{k\Delta_l})-\frac{\Delta_l}{2}h(x_{k\Delta_l})^*h(x_{k\Delta_l})\Big]\Big\}
$$
is simply a discretization of $Z_{T}$ (of the type of \cite{picard}).  Then set for $(t,\varphi)\in\mathbb{N}\times\mathcal{B}_b(\mathbb{R}^{d_x})$
\begin{eqnarray*}
\gamma_{t}^l(\varphi) & := & \overline{\mathbb{E}}\big[\varphi(X_t)Z_{t}^l(\widetilde{X}_0,\widetilde{X}_{\Delta_l},\dots,\widetilde{X}_{t-\Delta_l})|\mathcal{Y}_t\big] \\
\eta_{t}^l(\varphi) & := & \frac{\gamma_{t}^l(\varphi)}{\gamma_{t}^l(1)}.
\end{eqnarray*}
For notational convenience $\eta_0^l(dx) = \delta_{x_{*}}(dx)$. For 
$(l,p,t,\varphi)\in\mathbb{N}\times\{0,1,\dots\}\times\{\Delta_l,2\Delta_l,\dots,1-\Delta_l\}\times\mathcal{B}_b(\mathbb{R}^{d_x})$ one can also set
\begin{eqnarray*}
\gamma_{p+t}^l(\varphi) & := & \overline{\mathbb{E}}\Big[\varphi(X_{p+t})Z_{p}^l(\widetilde{X}_0,\widetilde{X}_{\Delta_l},\dots,\widetilde{X}_{p-\Delta_l})\Big(\prod_{k=0}^{t\Delta_l^{-1}-1}G_{p\Delta_l^{-1}+k}^l(\widetilde{X}_{p+k\Delta_l})\Big)\Big|\mathcal{Y}_{p+t}\Big] \\
\eta_{p+t}^l(\varphi) & := & \frac{\gamma_{p+t}^l(\varphi)}{\gamma_{p+t}^l(1)}
\end{eqnarray*}
where we define $Z_{0}^l(x_{-\Delta_l})=1$.

\subsection{Multilevel Monte Carlo}\label{sec:mlmc_rev}

In this section, to elaborate the methodology, we shall consider the estimation of $\eta_t^L(\varphi)$ for some fixed $(L,t,\varphi)\in\mathbb{N}_0\times\{\Delta_L,2\Delta_L,\dots\}\times\mathcal{B}_b(\mathbb{R}^d_{x})$. If it is possible, the Monte Carlo estimate of $\eta_t^L(\varphi)$ constitutes sampling $(X_t^{L,1},\dots,X_t^{L,N})$ i.i.d.~from $\eta_t^L$
and forming the estimate:
$$
\eta_{t,MC}^{L,N}(\varphi) := \frac{1}{N}\sum_{i=1}^N \varphi(x_t^{L,i}).
$$
In order to understand the error in estimation, one can consider the MSE:
$$
\mathbb{E}[(\eta_{t,MC}^{L,N}(\varphi)-\eta_{t}(\varphi))^2]
$$
where we note that the expectation operator $\mathbb{E}$ is that under $\mathbb{P}$. Then one has
$$
\mathbb{E}[(\eta_{t,MC}^{L,N}(\varphi)-\eta_{t}(\varphi))^2] \leq 2\Big(\mathbb{E}[(\eta_{t,MC}^{L,N}(\varphi)-\eta_{t}^L(\varphi))^2] + \mathbb{E}[(\eta_{t}^{L}(\varphi)-\eta_{t}(\varphi))^2]\Big).
$$
Now if, further, $\varphi\in\textrm{Lip}_{\|\cdot\|_2}(\mathbb{R}^{d_x})$ then using classical results in Monte Carlo estimation, (noting (D\ref{hyp_diff:1})) for the first expectation on the R.H.S.~and classical results on the bias of the Euler method (e.g.~\cite{picard}) for the second expectation on the R.H.S.~one has
\begin{equation}\label{eq:mc_mse_bd}
\mathbb{E}[(\eta_{t,MC}^{L,N}(\varphi)-\eta_{t}(\varphi))^2] \leq C\Big(\frac{1}{N} + \Delta_L\Big)
\end{equation}
where $C$ is a finite constant that may depend on $\varphi$ and $t$, but not $L$ nor $N$. For a given $\epsilon>0$ and assuming access to data with a high enough frequency, one can
choose $L$ so that $\Delta_L=\mathcal{O}(\epsilon^2)$ and choose $N=\mathcal{O}(\epsilon^{-2})$, so that the MSE is $\mathcal{O}(\epsilon^2)$. If the cost of simulation
of one sample is $\mathcal{O}(\Delta_L^{-1})$ as is often the case when working with Euler discretizations, then the cost to achieve this MSE is $\mathcal{O}(\epsilon^{-4})$ (the cost does not take into account the parameter $t$, which we shall not consider).

The MLMC method is associated to the telescoping sum identity:
\begin{equation}\label{eq:ml_id}
\eta_{t}^{L}(\varphi) = \eta_t^0(\varphi) + \sum_{l=1}^L[\eta_t^l- \eta_t^{l-1}](\varphi)
\end{equation}
where we are using the short-hand notation $[\eta_t^l- \eta_t^{l-1}](\varphi)=\eta_t^l(\varphi) - \eta_t^{l-1}(\varphi)$. We now explain how \eqref{eq:ml_id} can
be leveraged to reduce the cost to achieve an MSE of $\mathcal{O}(\epsilon^2)$.

The approach is to consider a method that can estimate $\eta_t^0(\varphi)$ and then, independently $[\eta_t^1- \eta_t^{0}](\varphi)$ and so on
until one independently estimates $[\eta_t^L- \eta_t^{L-1}](\varphi)$. The phrase independently, must be understood conditionally upon the data.
To estimate $\eta_t^0(\varphi)$, one can proceed just as above, when taking $L=0$. That is, one generates $(X_t^{0,1},\dots,X_t^{0,N_0})$ i.i.d.~from $\eta_t^0$
and forms the estimate $\eta_{t,MC}^{0,N_0}(\varphi) = \frac{1}{N_0}\sum_{i=1}^{N_0} \varphi(x_t^{0,i})$.

We now consider approximating $[\eta_t^l- \eta_t^{l-1}](\varphi)$ for $l\in\mathbb{N}$ fixed. We consider a (random) probability measure, $\check{\eta}_t^{l,l-1}$, on $(\mathbb{R}^{d_x}\times \mathbb{R}^{d_x},\mathsf{B}(\mathbb{R}^{d_x})\otimes \mathsf{B}(\mathbb{R}^{d_x}))$ such that for every $A\in \mathsf{B}(\mathbb{R}^{d_x})$ we have $\mathbb{P}-$almost surely
\begin{equation}\label{eq:coup_cond1}
\check{\eta}_t^{l,l-1}(A\times\mathbb{R}^{d_x}) = \eta_t^l(A) \quad\textrm{and}\quad \check{\eta}_t^{l,l-1}(\mathbb{R}^{d_x}\times A) = \eta_t^{l-1}(A).
\end{equation}
We will also \emph{assume} that  $\check{\eta}_t^{l,l-1}$ has the property that
\begin{equation}\label{eq:coup_cond}
\mathbb{E}\Big[
\int_{\mathbb{R}^{d_x}\times \mathbb{R}^{d_x}}\|x-\check{x}\|_2^2\check{\eta}_t^{l,l-1}(d(x,\check{x}))
\Big] \leq C \Delta_l^{\beta}
\end{equation}
where $C$ is a finite constant that does not depend upon $l$ and $\beta >0$ is a positive constant; we do not discuss the existence of $\check{\eta}_t^{l,l-1}$ as it is
used for purely illustrating the MLMC method, but such probabilities can exist; see for instance \cite{cpf_clt}. Note that the properties \eqref{eq:coup_cond1}-\eqref{eq:coup_cond} are the key for the method to be described - without them, it may not be of any practical use. The condition \eqref{eq:coup_cond} helps to achieve a variance reduction relative
to the Monte Carlo estimate as we will explain below.
Now one proceeds by sampling $((X_t^{l,1},\check{X}_t^{l-1,1}),\dots,(X_t^{l,N_l},\check{X}_t^{l-1,N_l}))$ i.i.d.~from $\check{\eta}_t^{l,l-1}$ and computing the estimate
$$
[\eta_t^l- \eta_t^{l-1}]_{MC}^{N_l}(\varphi) := \frac{1}{N_l}\sum_{i=1}^{N_l}[\varphi(x_t^{l,i}) - \varphi(\check{x}_t^{l-1,i})].
$$

The approximation of \eqref{eq:ml_id} is taken as:
$$
\eta_{t,MLMC}^{L,N_{0:L}}(\varphi) = \eta_{t,MC}^{0,N_0}(\varphi) + \sum_{l=1}^L [\eta_t^l- \eta_t^{l-1}]_{MC}^{N_l}(\varphi)
$$
where we stress that, conditional upon the data, the random variables $(\eta_{t,MC}^{0,N_0}(\varphi),[\eta_t^1- \eta_t^{0}]_{MC}^{N_1}(\varphi),\dots,[\eta_t^L- \eta_t^{L-1}]_{MC}^{N_L}(\varphi))$ are all independent and the notation $N_{0:L}=(N_0,\dots,N_L)$ is used. Now letting $\varphi\in\mathcal{B}_b(\mathbb{R}^{d_x})\cap\textrm{Lip}_{\|\cdot\|_2}(\mathbb{R}^{d_x})$, one can again consider the MSE
\begin{eqnarray*}
\mathbb{E}[(\eta_{t,MLMC}^{L,N_{0:L}}(\varphi)-\eta_{t}(\varphi))^2] & \leq & 2\Big(\mathbb{E}[(\eta_{t,MLMC}^{L,N_{0:L}}(\varphi)-\eta_{t}^L(\varphi))^2] + \mathbb{E}[(\eta_{t}^{L}(\varphi)-\eta_{t}(\varphi))^2]\Big) \\
& \leq & 2\Big(\mathbb{E}[(\eta_{t,MLMC}^{L,N_{0:L}}(\varphi)-\eta_{t}^L(\varphi))^2] + C\Delta_L\Big)
\end{eqnarray*}
where we have used the bias result that was applied to obtain \eqref{eq:mc_mse_bd}. Using standard results on sums of squares of random variables, along with the fact that, for $(l,q)\in\{1,\dots,L\}$, $l\neq q$
$$
\mathbb{E}\Big[\Big([\eta_t^l- \eta_t^{l-1}]_{MC}^{N_l}(\varphi) - [\eta_t^l- \eta_t^{l-1}](\varphi)\Big)
\Big([\eta_t^q- \eta_t^{q-1}]_{MC}^{N_q}(\varphi) - [\eta_t^q- \eta_t^{q-1}](\varphi)\Big)
\Big] = 0
$$
where we are using the conditional independence structure of the simulated random variables and \eqref{eq:coup_cond} one has 
$$
\mathbb{E}[(\eta_{t,MLMC}^{L,N_{0:L}}(\varphi)-\eta_{t}^L(\varphi))^2] \leq C\sum_{l=0}^{L}\frac{\Delta_l^{\beta}}{N_l}
$$
and thus
$$
\mathbb{E}[(\eta_{t,MLMC}^{L,N_{0:L}}(\varphi)-\eta_{t}(\varphi))^2] \leq C\Big(\sum_{l=0}^{L}\frac{\Delta_l^{\beta}}{N_l} + \Delta_L\Big)
$$
where, throughout, $C$ is a finite constant that does not depend upon $L$ or $N_{0:L}$. Suppose that $\beta$ as in \eqref{eq:coup_cond} is 1 and assume that the cost
of producing one sample from $\check{\eta}_t^{l,l-1}$ is $\mathcal{O}(\Delta_l^{-1})$. For a given $\epsilon>0$, using standard calculations (see e.g.~\cite{giles}), one can set 
$L=\mathcal{O}(|\log(\epsilon)|)$ (achieving a bias of $\mathcal{O}(\epsilon^2)$) and $N_l=\mathcal{O}(\epsilon^{-2}|\log(\epsilon)|\Delta_l)$, so that the MSE is $\mathcal{O}(\epsilon^2)$ and the cost to achieve this is $\mathcal{O}(\epsilon^{-2}\log(\epsilon)^2)$; a vast reduction over using the Monte Carlo method.

\subsection{Particle Filters}\label{sec:pf_rev}

The main objective of this section is to present a recursive and online method for approximating expectations $\eta_t^l(\varphi)$, where $(l,\varphi)\in\mathbb{N}_0\times\mathcal{B}_b(\mathbb{R}^d_{x})$ are fixed and $t\in\{\Delta_l,2\Delta_l,\dots\}$ is increasing. By `online' we mean that the approximation method will only use a computational cost that is fixed
for each $t$.

Particle filters are a simulation-based method that generates $N\in\mathbb{N}$ samples (or particles) in parallel. The algorithm constitutes two major steps, sampling and resampling. The sampling mechanism, is comprised of sampling the $N$ samples (conditionally) independently using the Euler dynamics. Then, as the Euler dynamics do not correspond to the true filter, one must correct for this fact, which is done using the operation of weighting and resampling. In this step all the samples will interact with each other.
PFs will produce estimates of $\eta_t^l(\varphi)$ which will converge almost surely as $N$ grows. In addition, one will also have an (almost surely) unbiased estimate of $\gamma_t^l(\varphi)$ - note that there is a discretization bias. See e.g.~\cite{crisan_bain,delmoral} for example.

The following concepts will be used in the algorithm to be described. Set, with $p\in\mathbb{N}_0$
$$
\mathbf{G}_p^l(x_{p:p+1-\Delta_l}) := \prod_{k=0}^{\Delta_l^{-1}-1}G_{p\Delta_l^{-1}+k}^l(x_{p+k\Delta_l})
$$
where we use the notation $x_{p:p+1-\Delta_l}=(x_p,x_{p+\Delta_l},x_{p+2\Delta_l},\dots,x_{p+1-\Delta_l})$.
This quantity will be used to correct samples generated from the Euler dynamics, to those which can be used to approximate the filter. Let $E_l=(\mathbb{R}^{d_x})^{\Delta_l^{-1}+1}$, and denote by 
$M^l:\mathbb{R}^{d_x}\rightarrow\mathcal{P}(E_l)$ the joint Markov transition of $(x_0,x_{\Delta_l},\dots,x_{1})$ defined via the Euler discretization \eqref{eq:disc_state} and a Dirac on a point $x\in\mathbb{R}^{d_x}$:  for $(x,\varphi)\in\mathbb{R}^{d_x}\times\mathcal{B}_b(E_l)$, 
$$
M^l(\varphi)(x) := \int_{E_l}\varphi(x_0,x_{\Delta_l},\dots,x_{1})\delta_x(dx_0)\Big[\prod_{k=1}^{\Delta_l^{-1}}\psi_{d_x}(x_{k\Delta_l};x_{(k-1)\Delta_l} + b(x_{(k-1)\Delta_l})\Delta_l ,a(x_{(k-1)\Delta_l})\Delta_l)\Big]d(x_{\Delta_l},\dots,x_{1}).
$$
This Markov transition kernel is the one that will be used to sample the process in-between weighting and resampling operations. 
We remark that the presence of the dirac mass $\delta_x$ is only used to keep consistency with \cite{high_freq_ml} as we will rely upon the theoretical results in that article. In practice
one does not need the dirac mass.
The algorithm is presented in details in Algorithm \ref{alg:pf}. In step 1.~of the algorithm, we generate $N$ samples independently from the Euler-dynamics. In step 2.~each sample is propagated by sampling from the probability measure \eqref{eq:sel_mutat}. This sampling encapsulates resampling and then sampling, first one computes the weight functions $\mathbf{G}_{p-1}^l$ and one selects a position $x_p^{l,i}$, from which to move the sample, with probability equal to
$$
\frac{\mathbf{G}_{p-1}^l(x_{p-1:p-\Delta_l}^{l,i})}{\sum_{j=1}^N \mathbf{G}_{p-1}^l(x_{p-1:p-\Delta_l}^{l,j})}.
$$
The sample is then moved according to the Markov kernel $M^l(x_p^{l,i},\cdot)$. If one wants to estimate $\eta_t^l(\varphi)$, $t\in\mathbb{N}$, then the estimate
\begin{equation}\label{eq:filt_pf}
\eta_{t,PF}^{l,N}(\varphi) := \sum_{i=1}^N \frac{\mathbf{G}_{t-1}^l(x_{t-1:t-\Delta_l}^{l,i})}{\sum_{j=1}^N \mathbf{G}_{t-1}^l(x_{t-1:t-\Delta_l}^{l,j})}\varphi(x_t^{l,i})
\end{equation}
is used and be computed after (the appropriate) step 2.~(or step 1.) in Algorithm \ref{alg:pf}. The estimate for valid non-integer $t$ is given in \cite{high_freq_ml}. An almost surely unbiased estimate of $\gamma_t^l(\varphi)$, $t\in\mathbb{N}$ is
$$
\gamma_{t,PF}^{l,N}(\varphi) := \Big[\prod_{p=0}^{t-2} \Big(\frac{1}{N}\sum_{i=1}^N \mathbf{G}_p^l(x_{p:p+1-\Delta_l}^{l,i})\Big)\Big]\Big(\frac{1}{N}\sum_{i=1}^N \mathbf{G}_{t-1}^l(x_{t-1:t-\Delta_l}^{l,i})\varphi(x_t^{l,i})\Big)
$$
and is again computed after (the appropriate) step 2.~(or step 1.) in Algorithm \ref{alg:pf}.

\begin{algorithm}[t]
\begin{enumerate}
\item{Initialize: For $i\in\{1,\dots,N\}$, generate $(x_0^{l,i},\dots,x_1^{l,i})$ from $M^l(x_{*},\cdot)$. Set $p=1$.}
\item{Update: For $i\in\{1,\dots,N\}$, generate $(x_p^{l,i},\dots, x_{p+1}^{l,i})$ from  
\begin{equation}\label{eq:sel_mutat}
\sum_{i=1}^N \frac{\mathbf{G}_{p-1}^l(x_{p-1:p-\Delta_l}^{l,i})}{\sum_{j=1}^N \mathbf{G}_{p-1}^l(x_{p-1:p-\Delta_l}^{l,j})}M^l(x_p^{l,i},\cdot).
\end{equation}
Set $p=p+1$ and return to the start of 2..}
\end{enumerate}
\caption{Particle Filter.}
\label{alg:pf}
\end{algorithm}

\subsection{Coupled Particle Filters}\label{sec:cpf_rev}

The main objective of this section is to present a recursive and online method for approximating expectations $[\eta_t^l-\eta_t^{l-1}](\varphi)$, where $(l,\varphi)\in\mathbb{N}\times\mathcal{B}_b(\mathbb{R}^{d_{x}})$ are fixed and $t\in\{\Delta_l,2\Delta_l,\dots\}$ is increasing. More precisely, we will seek to approximate expectations w.r.t.~probabilities 
$\check{\eta}_t^{l,l-1}$ as described in Section \ref{sec:mlmc_rev} with properties \eqref{eq:coup_cond1}-\eqref{eq:coup_cond}. In general, these probabilities are quite complex (see \cite{cpf_clt}), so we shall explain an approximation scheme that will correlate or couple the two steps in a particle filter. This approach will induce a sequence of targets $\check{\eta}_t^{l,l-1}$, $t\in\{\Delta_l,2\Delta_l,\dots\}$ which will possess the properties \eqref{eq:coup_cond1}-\eqref{eq:coup_cond}, but we will not discuss the details of these probabilities;
again information can be found in \cite{cpf_clt}.

The coupled particle filter will generate pairs of paths of the discretized diffusion, using $N$ samples simulated in parallel. The algorithm has two steps; coupled sampling and coupled resampling. The coupled sampling step constitutes sampling coupled Euler paths at the two levels of discretization. We describe the simulation of a Markov kernel $\check{P}^l:\mathbb{R}^{d_x}\times\mathbb{R}^{d_x}\rightarrow\mathcal{P}((\mathbb{R}^{d_x})^{\Delta_l^{-1}}\times (\mathbb{R}^{d_x})^{\Delta_{l-1}^{-1}})$ on paths $(x_{\Delta_l},\dots,x_1)$ and $(\check{x}_{\Delta_{l-1}},\dots,\check{x}_{1})$ (with initial points $(x,\check{x})\in\mathbb{R}^{2d_x}$) which provides a coupling of the Euler discretizations in Algorithm \ref{alg:coup_eul}. Let $\check{M}^l:\mathbb{R}^{d_x}\times\mathbb{R}^{d_x}\rightarrow\mathcal{P}(E_l\times E_{l-1})$ be a Markov kernel defined for $(u,\check{v},\varphi)\in\mathbb{R}^{d_x}\times\mathbb{R}^{d_x}\times\mathcal{B}_b(E_l\times E_{l-1})$ ($(u,\check{v})$ are the initial points of the kernel)
$$
\check{M}^l(\varphi)\Big((u,\check{v})\Big) := \int_{E_l\times E_{l-1}} \varphi(u^l,u^{l-1})\delta_u(dx_0^l)\delta_{\check{v}}(d\check{x}_0^{l-1}) \check{P}^l\Big((x_0^l, \check{x}_0^{l-1}) ,d((x_{\Delta_l}^l,\dots,x_1^l),(\check{x}_{\Delta_{l-1}}^{l-1},\dots,\check{x}_1^{l-1}))\Big)
$$
where we have used the notation $(u^l,\check{u}^{l-1}) = \big((x_0^l,x_{\Delta_l}^l,\dots,x_1^l),(\check{x}_0^{l-1},\check{x}_{\Delta_{l-1}}^{l-1},\dots,\check{x}_1^{l-1})\big)$.
This is the coupled simulation that we will use. Again, the dirac masses are only used for consistency with \cite{high_freq_ml} and are not needed in practice.

The coupled particle filter will generate pairs of trajectories $\big((x_p^{l,i},x_{p+\Delta_l}^{l,i},\dots,x_{p+1}^{l,i}),(\check{x}_p^{l-1,i},\check{x}_{p+\Delta_{l-1}}^{l-1,i},\dots,\check{x}_{p+1}^{l-1,i})\big)$ with $p\in\mathbb{N}_0$. The idea will be to resample these trajectories at times $1,2,\dots$ so that the trajectory $(x_p^{l,i},x_{p+\Delta_l}^{l,i},\dots,x_{p+1}^{l,i})$ is resampled using the probability distribution on $\{1,\dots,N\}$ as
$$
\frac{\mathbf{G}_{p}^l(x_{p:p+1-\Delta_l}^{l,i})}{\sum_{j=1}^N \mathbf{G}_{p}^l(x_{p:p+1-\Delta_l}^{l,j})}
$$
and  the trajectory $(\check{x}_p^{l-1,i},\check{x}_{p+\Delta_{l-1}}^{l-1,i},\dots,\check{x}_{p+1}^{l-1,i})$ is resampled using the probability distribution on $\{1,\dots,N\}$ as
$$
\frac{\mathbf{G}_{p}^{l-1}(\check{x}_{p:p+1-\Delta_l}^{l-1,i})}{\sum_{j=1}^N \mathbf{G}_{p}^{l-1}(\check{x}_{p:p+1-\Delta_l}^{l-1,j})}
$$
but that the sampling of the pair of indices on $\{1,\dots,N\}$ is not independent. The reason for this is that one would like to obtain a property such as \eqref{eq:coup_cond} for the limiting distribution $\check{\eta}^{l,l-1}_t$, which is seldom possible if the resampling operation is independent between pairs of trajectories (see e.g.~\cite{cpf_clt}). In Algorithm \ref{alg:max_coup} we describe one way to achieve this (re)sampling for two generic probability mass functions on $\{1,\dots,N\}$; the coupling is called the maximal coupling.

\begin{algorithm}[t]
\begin{enumerate}
\item{Generate $(V_{\Delta_l},V_{2\Delta_l}, \dots,V_1)$, where, for $k\in\{1,\dots,\Delta_l^{-1}\}$, $V_{k\Delta_l} \stackrel{\textrm{i.i.d.}}{\sim}\mathcal{N}_{d_x}(0,\Delta_l)$.}
\item{Set $x_0=x$ and compute the recursion: $x_{k\Delta_l}  =  x_{(k-1)\Delta_l} + b(x_{(k-1)\Delta_l})\Delta_l + \sigma(x_{(k-1)\Delta_l})V_{k\Delta_l}$, 
$k\in\{1,\dots,\Delta_l^{-1}\}$.}
\item{Set $\check{x}_0=\check{x}$ and compute the recursion: $\check{x}_{k\Delta_{l-1}}  =  \check{x}_{(k-1)\Delta_{l-1}} + b(\check{x}_{(k-1)\Delta_{l-1}})\Delta_{l-1} + \sigma(\check{x}_{(k-1)\Delta_{l-1}})[V_{(2k-1)\Delta_l} + 
V_{2k\Delta_l}]$, $k\in\{1,\dots,\Delta_{l-1}^{-1}\}$.}
\end{enumerate}
\caption{Simulating Coupled Euler Discretizations. The initial point of the two Euler paths is $(x,\check{x})\in\mathbb{R}^{2d_x}$.}
\label{alg:coup_eul}
\end{algorithm}

\begin{algorithm}[t]
\begin{enumerate}
\item{Input: Two probability mass functions (PMFs) $(r_1^1,\dots,r_1^N)$ and $(r_2^1,\dots,r_2^N)$ on $\{1,\dots,N\}$.}
\item{Generate $U\sim\mathcal{U}_{[0,1]}$.}
\item{If $U<\sum_{i=1}^N \min\{r_1^i,r_2^i\}=:\bar{r}$ then generate $i\in\{1,\dots,N\}$ according to the probability mass function:
$$
r_3^i = \frac{1}{\bar{r}} \min\{r_1^i,r_2^i\}
$$
and set $j=i$.}
\item{Otherwise generate $i\in\{1,\dots,N\}$ and $j\in\{1,\dots,N\}$ independently according to the probability mass functions 
$$
r_4^i = \frac{1}{1-\bar{r}} (r_1^i - \min\{r_1^i,r_2^i\})
$$
and
$$
r_5^j = \frac{1}{1-\bar{r}} (r_2^j - \min\{r_1^j,r_2^j\})
$$
respectively.
}
\item{Output: $(i,j)\in\{1,\dots,N\}^2$. $i$, marginally has PMF $r_1^i$ and $j$, marginally has PMF $r_2^j$.}
\end{enumerate}
\caption{Simulating a Maximal Coupling of Two Probability Mass Functions on $\{1,\dots,N\}$.}
\label{alg:max_coup}
\end{algorithm}

\begin{algorithm}[t]
\begin{enumerate}
\item{Initialize: For $i\in\{1,\dots,N\}$, generate $\big((x_0^{l,i},\dots,x_1^{l,i}),(\check{x}_0^{l-1,i},\dots,\check{x}_1^{l-1,i})\big)$ from $\check{M}^l\Big((x_{*},x_{*}),\cdot\Big)$. Set $p=1$.}
\item{Update: For $i\in\{1,\dots,N\}$, generate two indices $(s_{p-1}^{l,i},\check{s}_{p-1}^{l-1,i})$ by using Algorithm \ref{alg:max_coup} with input probability mass functions
$$
\Bigg(\frac{\mathbf{G}_{p-1}^l(x_{p-1:p-\Delta_l}^{l,1})}{\sum_{j=1}^N \mathbf{G}_{p-1}^l(x_{p-1:p-\Delta_l}^{l,j})},\dots,\frac{\mathbf{G}_{p-1}^l(x_{p-1:p-\Delta_l}^{l,N})}{\sum_{j=1}^N \mathbf{G}_{p-1}^l(x_{p-1:p-\Delta_l}^{l,j})}\Bigg)
$$
and
$$
\Bigg(\frac{\mathbf{G}_{p-1}^{l-1}(\check{x}_{p-1:p-\Delta_{l-1}}^{l-1,1})}{\sum_{j=1}^N \mathbf{G}_{p-1}^{l-1}(\check{x}_{p-1:p-\Delta_{l-1}}^{l-1,j})},\dots,\frac{\mathbf{G}_{p-1}^{l-1}(\check{x}_{p-1:p-\Delta_{l-1}}^{l-1,N})}{\sum_{j=1}^N \mathbf{G}_{p-1}^{l-1}(\check{x}_{p-1:p-\Delta_{l-1}}^{l-1,j})}\Bigg).
$$
Then generate $\big((x_p^{l,i},\dots,x_{p+1}^{l,i}),(\check{x}_{p}^{l-1,i},\dots,\check{x}_{p+1}^{l-1,i})\big)$ from $\check{M}^l\Big((x_{p}^{l,s_{p-1}^{l,i}},\check{x}_{p}^{l-1,\check{s}_{p-1}^{l-1,i}}),\cdot\Big)$. 
Set $p=p+1$ and return to the start of 2..}
\end{enumerate}
\caption{Coupled Particle Filter.}
\label{alg:cpf}
\end{algorithm}

Given the ideas in Algorithms \ref{alg:coup_eul} and \ref{alg:max_coup} we can now detail the coupled particle filter, which is given in Algorithm \ref{alg:cpf}. One can compute 
an estimate of $[\eta_t^l-\eta_t^{l-1}](\varphi)$ for $t\in\mathbb{N}$ as
\begin{equation}\label{eq:diff_cpf}
[\eta_t^l-\eta_t^{l-1}]^N_{CPF}(\varphi) := \sum_{i=1}^N \frac{\mathbf{G}_{t-1}^l(x_{t-1:t-\Delta_l}^{l,i})}{\sum_{j=1}^N \mathbf{G}_{t-1}^l(x_{t-1:t-\Delta_l}^{l,j})}\varphi(x_t^{l,i}) -
\sum_{i=1}^N \frac{\mathbf{G}_{t-1}^{l-1}(\check{x}_{t-1:t-\Delta_{l-1}}^{l-1,i})}{\sum_{j=1}^N \mathbf{G}_{t-1}^{l-1}(\check{x}_{t-1:t-\Delta_{l-1}}^{l-1,j})}\varphi(\check{x}_t^{l-1,i})
\end{equation}
 after (the appropriate) step 2. (or step 1.)~in Algorithm \ref{alg:cpf}. In addition, one can compute an almost surely unbiased estimate (we will prove this in Section \ref{sec:mlpf_rev}) of $[\gamma_t^l-\gamma_t^{l-1}](\varphi)$ for $(t,\varphi)\in\mathbb{N}\times\mathcal{B}_b(\mathbb{R}^{d_x})$ as 
 \begin{eqnarray*}
 [\gamma_t^l-\gamma_t^{l-1}]^N_{CPF}(\varphi) & := &
 \Big[\prod_{p=0}^{t-2} \Big(\frac{1}{N}\sum_{i=1}^N \mathbf{G}_p^l(x_{p:p+1-\Delta_l}^{l,i})\Big)\Big]\Big(\frac{1}{N}\sum_{i=1}^N \mathbf{G}_{t-1}^l(x_{t-1:t-\Delta_l}^{l,i})\varphi(x_t^{l,i})\Big) - \\ & & 
 \Big[\prod_{p=0}^{t-2} \Big(\frac{1}{N}\sum_{i=1}^N \mathbf{G}_p^{l-1}(\check{x}_{p:p+1-\Delta_{l-1}}^{l-1,i})\Big)\Big]\Big(\frac{1}{N}\sum_{i=1}^N \mathbf{G}_{t-1}^{l-1}(\check{x}_{t-1:t-\Delta_{l-1}}^{l-1,i})\varphi(\check{x}_t^{l-1,i})\Big)
 \end{eqnarray*}
  after (the appropriate) step 2. (or step 1.)~in Algorithm \ref{alg:cpf}.

\subsection{Multilevel Particle Filter}\label{sec:mlpf_rev}

To summarize, one can implement multilevel estimates of the filter and the solution to Zakai's equation in the following manner, which we call the MLPF.
\begin{enumerate}
\item{Run the particle filter in Algorithm \ref{alg:pf} for $l=0$ and $N=N_0$ and up-to the desired time.}
\item{For $l\in\{1,\dots,L\}$, independently of all of other sampling, run Algorithm \ref{alg:cpf} with $N=N_l$ and up-to the desired time.}
\end{enumerate}
The MLPF estimate of $\eta_t^L(\varphi)$, for $(t,\varphi)\in(\mathbb{N},\mathcal{B}_b(\mathbb{R}^{d_x}))$, is then
\begin{equation}\label{eq:filt_mlpf}
\eta_{t,MLPF}^{L,N_{0:L}}(\varphi) := \eta_{t,PF}^{0,N_0}(\varphi) + \sum_{l=1}^L[\eta_t^l-\eta_t^{l-1}]^{N_l}_{CPF}(\varphi)
\end{equation} 
where the first term on the R.H.S.~is defined in \eqref{eq:filt_pf} and the summands on the R.H.S.~are defined in \eqref{eq:diff_cpf}.  
For $\epsilon>0$ given and $\sigma$ is a non-constant function, consider \eqref{eq:filt_mlpf}, when one chooses $L=\mathcal{O}(|\log(\epsilon|))$, $N_l=\mathcal{O}(\epsilon^{-2}\Delta_l^{-1/4}\Delta_l^{3/4})$. In \cite{high_freq_ml}, it is proved under assumptions, one can achieve a MSE of $\mathcal{O}(\epsilon^2)$ for a cost of $\mathcal{O}(\epsilon^{-3})$. This is better than using a particle filter to approximate $\eta_t^L(\varphi)$ which has a MSE of $\mathcal{O}(\epsilon^2)$ for a cost of $\mathcal{O}(\epsilon^{-4})$; see \cite{high_freq_ml}.

To estimate the solution of Zaki's equation, one can use an approach in \cite{mlpf_nc} that was not considered in \cite{high_freq_ml}:
\begin{equation}\label{eq:zakai_ml}
\gamma_{t,MLPF}^{L,N_{0:L}}(\varphi) := \gamma_{t,PF}^{0,N_0}(\varphi) + \sum_{l=1}^L[\gamma_t^l-\gamma_t^{l-1}]^{N_l}_{CPF}(\varphi).
\end{equation}
This estimator will be both mathematically analyzed in this article and implemented in our numerical examples. In terms of the former, we will compute meaningful bounds on the MSE in terms of $N_{0:L}$ and $(\Delta_0,\dots,\Delta_L)$. Before continuing, the following result will be very useful in the subsequent discussion; the proof can be found in Appendix \ref{app:proofs}, although it is essentially a direct corollary of \cite[Theorem 7.4.2.]{delm:04}.

\begin{prop}\label{prop:ub_nc}
Assume (D1). Then we have for any $(t,L,,N_{0:L},\varphi)\in\mathbb{N}\times\mathbb{N}\times\mathbb{N}^{L+1}\times\mathcal{B}_b(\mathbb{R}^{d_x})$ that almost surely:
$$
\mathbb{\overline{E}}[\gamma_{t,MLPF}^{L,N_{0:L}}(\varphi)|\mathcal{Y}_t] = \gamma_t^L(\varphi).
$$
\end{prop}

\begin{rem}
The proof shows that for any $(t,N_0,\varphi)\in\mathbb{N}\times\mathbb{N}\times\mathcal{B}_b(\mathbb{R}^{d_x})$, almost surely
$$
\mathbb{\overline{E}}[\gamma_{t,PF}^{0,N_0}(\varphi)|\mathcal{Y}_t] = \gamma_t^0(\varphi)
$$
and for any $(t,l,N_l,\varphi)\in\mathbb{N}\times\mathbb{N}\times\mathbb{N}\times\mathcal{B}_b(\mathbb{R}^{d_x})$ almost surely that
$$
\mathbb{\overline{E}}[[\gamma_t^l-\gamma_t^{l-1}]^{N_l}_{CPF}(\varphi)|\mathcal{Y}_t] = [\gamma_t^l-\gamma_t^{l-1}](\varphi).
$$
\end{rem}

\begin{rem}
The support of $t$ can be increased; see \cite{high_freq_ml} for details.
\end{rem}

\section{Method for Unbiased Estimation}\label{sec:method}

Our objective is to compute an almost surely unbiased estimator of $\gamma_t(\varphi)$, where we shall constrain ourselves to the case that $(t,\varphi)\in\mathbb{N}\times\mathcal{B}_b(\mathbb{R}^{d_x})\cap\textrm{Lip}_{\|\cdot\|_2}(\mathbb{R}^{d_x})$. We note that the constraints on $t$ can easily be relaxed (see \cite{high_freq_ml}), but, for notational ease, we
maintain the convention that $t\in\mathbb{N}$. The approach that we use to construct our estimators is to combine the PF and CPF methodology considered in Sections \ref{sec:pf_rev} and \ref{sec:cpf_rev} along with the unbiased estimation methods of \cite{rhee} (see also \cite{mcl, matti}). To achieve this, one needs the following scenario:
\begin{quote}
There exist a sequence of independent random variables $(\Psi_t^l(\varphi))_{l\geq 0}$ such that for any $(t,\varphi)\in\mathbb{N}\times\mathcal{B}_b(\mathbb{R}^{d_x})$,  almost surely:
$$
\mathbb{\overline{E}}[\Psi_t^0(\varphi)|\mathcal{Y}_t] = \gamma_t^0(\varphi)
$$
and for any  $(t,l,\varphi)\in\mathbb{N}\times\mathbb{N}\times\mathcal{B}_b(\mathbb{R}^{d_x})$,  almost surely:
$$
\mathbb{\overline{E}}[\Psi_t^l(\varphi)|\mathcal{Y}_t] = \gamma_t^l(\varphi).
$$
\end{quote}
The existence of such random variables are assured, by the discussion in Section \ref{sec:mlpf_rev}. More precisely, we will set
$(t,N,\varphi)\in\mathbb{N}\times\mathbb{N}\times\mathcal{B}_b(\mathbb{R}^{d_x})\cap\textrm{Lip}_{\|\cdot\|_2}(\mathbb{R}^{d_x})$,
\begin{equation}\label{eq:psi0_def}
\Psi_t^0(\varphi) := \gamma_{t,PF}^{0,N}(\varphi)
\end{equation}
and for any  $(t,l,N,\varphi)\in\mathbb{N}\times\mathbb{N}\times\mathbb{N}\times\mathcal{B}_b(\mathbb{R}^{d_x})\cap\textrm{Lip}_{\|\cdot\|_2}(\mathbb{R}^{d_x})$,  almost surely:
\begin{equation}\label{eq:psil_def}
\Psi_t^l(\varphi) = [\gamma_t^l-\gamma_t^{l-1}]^{N}_{CPF}(\varphi).
\end{equation}
In the construction to be considered, we choose $N$ independently of the value of $l$.
We also require that
$$
\lim_{l\rightarrow\infty}\mathbb{\overline{E}}[\gamma_t^l(\varphi)] = \mathbb{\overline{E}}[\gamma_t(\varphi)].
$$
but this is assured by Lemma \ref{lem:disc_zakai} 2.~(in the appendix).

Now, let $P(l)$ be a positive probability mass function on $\{0,1,\dots\}$ (independent of $(\Psi_t^l(\varphi))_{l\geq 0}$) and $Q(l)=\sum_{q=l}^{\infty} P(q)$, $l\in\{0,1,\dots,\}$, then we propose two estimators whose calculation is detailed in Algorithms \ref{alg:st} and \ref{alg:cs} and the estimators are expressed in equations \eqref{eq:st_0}-\eqref{eq:cs}. 
Algorithm \ref{alg:st} constructs the single term (ST) estimator in \cite{rhee} and Algorithm \ref{alg:cs} the coupled sum (CS) estimator of the same article.
We will discuss how to choose $P(l)$ in the next section.
In practice one would repeat  Algorithms \ref{alg:st} and \ref{alg:cs} independently $M$ times (in parallel) and, denoting the $i^{th}-$estimator $\gamma_{t,ST}(\varphi)^i$ (resp.~($\gamma_{t,CS}(\varphi)^i$) one would use the estimate
$$
\gamma_{t,ST}^M(\varphi) = \frac{1}{M}\sum_{i=1}^M \gamma_{t,ST}(\varphi)^i \quad\textrm{resp.}\quad
\gamma_{t,CS}^M(\varphi) = \frac{1}{M}\sum_{i=1}^M \gamma_{t,CS}(\varphi)^i.
$$
As noted, the implementation is straight-forwardly made parallel and is one of the attractions of the approach. The main issue is then to verify that the estimators 
$\gamma_{t,ST}(\varphi)$ and $\gamma_{t,CS}(\varphi)$ are almost surely unbiased with finite variance and to discuss the cost for achieving this; we consider this in the next section.

\begin{algorithm}[b]
\begin{enumerate}
\item{Generate $L\in\{0,1,\dots\}$ using $P(\cdot)$.}
\item{If $L=0$ run the particle filter, as in Algorithm \ref{alg:pf}, with $l=0$ and $N$ samples. Return the estimator, for each $t\in\mathbb{N}$
\begin{equation}\label{eq:st_0}
\gamma_{t,ST}(\varphi) = \frac{\Psi_t^0(\varphi)}{P(0)}.
\end{equation}
}
\item{Otherwise run the coupled particle filter, as in Algorithm \ref{alg:cpf}, with level $l$ and $N$ samples. Return the estimator, for each $t\in\mathbb{N}$
\begin{equation}\label{eq:st}
\gamma_{t,ST}(\varphi) = \frac{\Psi_t^l(\varphi)}{P(l)}.
\end{equation}
}
\end{enumerate}
\caption{The Single Term Unbiased Estimator of $\gamma_t(\varphi)$.}
\label{alg:st}
\end{algorithm}

\begin{algorithm}[t]
\begin{enumerate}
\item{Generate $L\in\{0,1,\dots\}$ using $P(\cdot)$. Run the particle filter, as in Algorithm \ref{alg:pf}, with $l=0$ and $N$ samples.}
\item{If $L=0$ return the estimator, for each $t\in\mathbb{N}$
\begin{equation}\label{eq:cs_0}
\gamma_{t,CS}(\varphi) = \Psi_t^0(\varphi).
\end{equation}
}
\item{Otherwise, for each $l\in\{1,\dots,L\}$ independently run the coupled particle filter, as in Algorithm \ref{alg:cpf}, with level $l$ and $N$ samples. Return the estimator, for each $t\in\mathbb{N}$
\begin{equation}\label{eq:cs}
\gamma_{t,CS}(\varphi) = \Psi_t^0(\varphi) + \sum_{l=1}^L \frac{\Psi_t^l(\varphi)}{Q(l)}
\end{equation}
}
\end{enumerate}
\caption{The Coupled Sum Unbiased Estimator of $\gamma_t(\varphi)$.}
\label{alg:cs}
\end{algorithm}

\section{Theoretical Results}\label{sec:theory}

\subsection{Multilevel and Unbiased Methods}

We now present several theoretical results which will allow us to understand the utility of our suggested approaches.
We begin with a result for the multilevel estimate for the solution to Zakai, in \eqref{eq:zakai_ml}.
Throughout our proofs $C$ is a finite constant whose value may change on appearance and does not depend upon $l$ nor $N$.
Propositions or Lemmata with a numbering A can be found in the appendix.

\begin{prop}\label{prop:mlpf_nc}
Assume (D\ref{hyp_diff:1}). Then for any $t\in\mathbb{N}$, there exists a $C<+\infty$ such that
for any $(L,N_{0:L},\varphi)\in\mathbb{N}\times \mathbb{N}^{L+1}\times\mathcal{B}_b(\mathbb{R}^{d_x})\cap\textrm{\emph{Lip}}_{\|\cdot\|_2}(\mathbb{R}^{d_x})$
$$
\mathbb{\overline{E}}[
(\gamma_{t,MLPF}^{L,N_{0:L}}(\varphi)-\gamma_t(\varphi))^2] \leq C(\|\varphi\|+\|\varphi\|_{\textrm{\emph{Lip}}})^2\Big(\sum_{l=0}^L \frac{\Delta_l^{1/2}}{N}+\Delta_L\Big).
$$
\end{prop}

\begin{proof}
By the $C_2-$inequality
$$
\mathbb{\overline{E}}[
(\gamma_{t,MLPF}^{L,N_{0:L}}(\varphi)-\gamma_t(\varphi))^2] \leq 
2\Big(
\mathbb{\overline{E}}[
(\gamma_{t,MLPF}^{L,N_{0:L}}(\varphi)-\gamma_{t}^{L}(\varphi))^2]  + 
\mathbb{\overline{E}}[
(\gamma_{t}^{L}(\varphi)-\gamma_t(\varphi))^2] \Big).
$$
For the second expectation on the R.H.S.~one can apply Lemma \ref{lem:disc_zakai} 2.~(in the appendix) to give
$$
\mathbb{\overline{E}}[
(\gamma_{t,MLPF}^{L,N_{0:L}}(\varphi)-\gamma_t(\varphi))^2] \leq C\Big(
\mathbb{\overline{E}}[
(\gamma_{t,MLPF}^{L,N_{0:L}}(\varphi)-\gamma_{t}^{L}(\varphi))^2]  + (\|\varphi\|+\|\varphi\|_{\textrm{Lip}})^2\Delta_L
\Big).
$$
For the first expectation on the R.H.S.~using $\gamma_{t}^{L}(\varphi) = \gamma_{t}^{0}(\varphi) + \sum_{l=1}^L[\gamma_{t}^{l}-\gamma_{t}^{l-1}](\varphi)$, along with
Proposition \ref{prop:ub_nc}
we have
$$
\mathbb{\overline{E}}[
(\gamma_{t,MLPF}^{L,N_{0:L}}(\varphi)-\gamma_t(\varphi))^2] \leq C\Big(
 \mathbb{\overline{E}}[[\gamma_{t,PF}^{0,N_0}-\gamma_t^0](\varphi)^2] +
\sum_{l=1}^L \mathbb{\overline{E}}[(
[\gamma_t^l-\gamma_t^{l-1}]^{N_l}_{CPF}(\varphi) -  [\gamma_t^l-\gamma_t^{l-1}](\varphi))^2] +
(\|\varphi\|+\|\varphi\|_{\textrm{Lip}})^2\Delta_L
\Big).
$$
Then applying Proposition \ref{prop:cpf_bound} along with Remark \ref{rem:pf_bound} one can conclude.
\end{proof}

We now consider the case of our suggested unbiased estimators.

\begin{prop}\label{prop:ub_sm}
Assume (D\ref{hyp_diff:1}). Then for any $t\in\mathbb{N}$, there exists a $C<+\infty$ such that for any $(N,\varphi)\in\mathbb{N}\times\mathcal{B}_b(\mathbb{R}^{d_x})\cap\textrm{\emph{Lip}}_{\|\cdot\|_2}(\mathbb{R}^{d_x})$
\begin{eqnarray*}
\mathbb{\overline{E}}[\gamma_{t,ST}(\varphi)^2]
& \leq & C(\|\varphi\|+\|\varphi\|_{\textrm{\emph{Lip}}})^2\Big(\sum_{l=0}^{\infty} \frac{1}{P(l)}\Big\{\frac{\Delta_l^{1/2}}{N}+\Delta_l\Big\}\Big)\\
\mathbb{\overline{E}}[\gamma_{t,CS}(\varphi)^2]
& \leq & C(\|\varphi\|+\|\varphi\|_{\textrm{\emph{Lip}}})^2\Big(\sum_{l=0}^{\infty} \frac{1}{Q(l)}\Big\{\frac{\Delta_l^{1/2}}{N}+\Delta_l\Big\}\Big).
\end{eqnarray*}
\end{prop}

\begin{proof}
For the first inequality, we have by (e.g.) \cite[Theorem 3]{matti} that
$$
\mathbb{\overline{E}}[\gamma_{t,ST}(\varphi)^2] = \sum_{l=0}^{\infty} \frac{1}{P(l)}\mathbb{\overline{E}}[\Psi_t^l(\varphi)^2].
$$
Noting \eqref{eq:psi0_def}, \eqref{eq:psil_def} and adding and subtracting $[\gamma_t^l-\gamma_t^{l-1}](\varphi)$ (resp.~$\gamma_t^0(\varphi)$) inside the square of the expectation when the index in the sum is greater than or equal to 1 (resp.~0) and using the $C_2-$inequality
\begin{eqnarray*}
\mathbb{\overline{E}}[\gamma_{t,ST}(\varphi)^2] & \leq &
\frac{C}{P(0)} \Big(
\mathbb{\overline{E}}[[\gamma_{t,PF}^{0,N}-\gamma_t^0](\varphi)^2] 
 + 
\mathbb{\overline{E}}[\gamma_t^0(\varphi)^2]
 \Big)
+ \\ & &C\Big(\sum_{l=1}^{\infty} \frac{1}{P(l)}
\mathbb{\overline{E}}[(
[\gamma_t^l-\gamma_t^{l-1}]^{N}_{CPF}(\varphi) -  [\gamma_t^l-\gamma_t^{l-1}](\varphi))^2] + \mathbb{\overline{E}}[[\gamma_t^l-\gamma_t^{l-1}](\varphi)^2]
\Big).
\end{eqnarray*}
Then applying Proposition \ref{prop:cpf_bound} (summands on the R.H.S.~when $l\geq 1$) along with Remark \ref{rem:pf_bound} (the summand on the R.H.S.~when $l=0$) as well as Lemma \ref{lem:disc_zakai} 1.~and that 
$\mathbb{\overline{E}}[\gamma_t^0(\varphi)^2]\leq C(\|\varphi\|+\|\varphi\|_{\textrm{Lip}})^2$ allows one to obtain the quoted bound.
For the second inequality, we have by (e.g.) \cite[Theorem 5]{matti} that 
$$
\mathbb{\overline{E}}[\gamma_{t,CS}(\varphi)^2] = 
\sum_{l=0}^{\infty} \frac{1}{Q(l)}\Big(\mathbb{\overline{E}}[\Psi_t^l(\varphi)^2]-\mathbb{\overline{E}}[\Psi_t^l(\varphi)]^2 + 
(\mathbb{\overline{E}}[\gamma_{t}^{l-1}(\varphi)] - \mathbb{\overline{E}}[\gamma_{t}(\varphi)])^2 -
(\mathbb{\overline{E}}[\gamma_{t}^{l}(\varphi)] - \mathbb{\overline{E}}[\gamma_{t}(\varphi)])^2 
\Big)
$$
where we take $\gamma_{t}^{-1}(\varphi)=0$. One can use a similar argument to the first inequality to conclude.
\end{proof}

\begin{rem}
Our estimators possess
the unbiased property:
$$
\mathbb{\overline{E}}[\gamma_{t,ST}(\varphi)] = \mathbb{\overline{E}}[\gamma_{t,CS}(\varphi)] = \mathbb{\overline{E}}[\gamma_{t}(\varphi)].
$$
It is also straightforward to deduce that for any continuous, bounded and real-valued function on a trajectory, $\phi$, that
$$
\mathbb{\overline{E}}\Big[\gamma_{t,ST}(\varphi)\phi\Big(\{Y_s\}_{s\in[0,t]}\Big)\Big] = \mathbb{\overline{E}}\Big[\gamma_{t,CS}(\varphi)\phi\Big(\{Y_s\}_{s\in[0,t]}\Big)\Big] = 
\mathbb{\overline{E}}\Big[\gamma_{t}(\varphi)\phi\Big(\{Y_s\}_{s\in[0,t]}\Big)\Big]
$$
that is, $\mathbb{\overline{E}}[\gamma_{t,ST}(\varphi)|\mathcal{Y}_t]$ and $\mathbb{\overline{E}}[\gamma_{t,CS}(\varphi)|\mathcal{Y}_t]$ are versions of $\gamma_t(\varphi)$
and hence almost surely unbiased estimators of $\gamma_t(\varphi)$.
If one chooses $P(l)$ so that the two bounds in Proposition \ref{prop:ub_sm} are finite then our suggested estimators have finite variance due to the unbiased property and  
$\mathbb{\overline{E}}[\gamma_t(\varphi)]^2$ being finite.
\end{rem}

\begin{rem}
It is remarked that all of the bounds depend upon $t$ and one would expect using a more detailed (and more complicated) approach, the bounds can be made $t$ independent.
See \cite{cpf_clt} for some work in that direction.
\end{rem}

\subsection{Discussion of Costs}\label{sec:cost_disc}

We first begin with the case of the MLPF estimate. This follows the analysis in \cite{high_freq_ml} as follows. Let $\epsilon>0$ be arbitrary, to obtain a bound on the MSE, in Proposition \ref{prop:mlpf_nc}, of $\mathcal{O}(\epsilon^2)$ one can choose $L$ so that $\Delta_L=\mathcal{O}(\epsilon^2)$. Then setting $N_l=\mathcal{O}(\epsilon^{-2}\Delta_L^{-1/4}\Delta_{l}^{3/4})$, the upper-bound
in Proposition \ref{prop:mlpf_nc} is $\mathcal{O}(\epsilon^2)$. The associated cost to achieve this MSE is $\mathcal{O}(\sum_{l=0}^L \Delta_l^{-1}N_l)=\mathcal{O}(\epsilon^{-3})$. 
Using Remark \ref{rem:pf_bound} and Lemma \ref{lem:disc_zakai} 2., one can show that using an estimator such as $\gamma_{t,PF}^{L,N}(\varphi)$ (with $L$ chosen so that $\Delta_L=\mathcal{O}(\epsilon^2)$) one would need a cost of $\mathcal{O}(\epsilon^{-4})$ to obtain a MSE
of $\mathcal{O}(\epsilon^2)$. We remark that choosing $\Delta_L=\mathcal{O}(\epsilon^2)$ requires one to have access to data which is sufficiently frequently observed and so this
methodology is primarily useful for the observation of ultra-high frequency data.

In the case of the randomized estimators, we focus on the case of a single term estimator $\gamma_{t,ST}^M(\varphi)$. We note that in order to choose $P(l)$ so that the variance of the estimator is finite and that the expected cost, $tN\sum_{l=0}^{\infty} \Delta_l^{-1} P(l)$, is finite, is typically not possible. However, one can consider the context of \cite[Proposition 5]{rhee}. There, the authors show that if $P(l)\propto \Delta_l^{1/2+\alpha}(l+1)\log_2(l+2)^2$, for some $0<\alpha<1/2$, then with high probability, the cost to be within $\epsilon$ of
$\gamma_t(\varphi)$ is $\mathcal{O}(\epsilon^{-3}|\log_2(\epsilon)|^{3+\beta})$ for some $\beta>0$. This means that the unbiased estimator will cost a little more, in general, than using the MLPF approach. We remark however, that the unbiased methodology is embarrassingly parallel, whereas the MLPF has limitations to the extent to which it can be parallelized. Note that choosing $N$ to be dependent on $l$ can reduce the variance, whilst increasing the cost. We have found the strategy of fixing $N$ (in terms of $l$) and using $P(l)$ as detailed works well in practice and maintain this convention in the next section. It should be remarked that in principle, as for the MLPF approach, one must have access to very high
(in fact arbitrarily high) frequency data. In practice, however, one cannot run the algorithm beyond say $l=50$ and one has to truncate the estimate at the cost of some bias.

\section{Numerical Simulations}\label{sec:numerics}

\subsection{Models}
The numerical performance of the unbiased single-term (ST) and the coupled-sum (CS) estimators will be illustrated here with four different examples of one-dimensional diffusion processes; we set $d_y=d_x=1$.
 We also compare the performance of these estimators with a MLPF. Under the new measure $\overline{\mathbb{P}}$ the data $\{Y_t\}_{t\geq 0}$ is a standard Brownian motion that is independent of the process $\{X_t\}_{t\geq 0}$; the data are simulated from the observation process (i.e.~standard Brownian motion). 
For all the examples below we take $h(x)=x$, $t=50$ and $\varphi(x) = x$. The diffusion processes considered are
\begin{enumerate}
\item \textbf{Ornstein-Uhlenbeck (OU) Process:}\\
\\
In this example $b(x)=-x$, $\sigma(x)=0.5$, $x_*=0$.
\item \textbf{Langevin Stochastic Differential Equation:}\\
\\
For this process we take $b(x)=\frac{1}{2}\frac{d}{dx} \left(\log f(x)\right)$, where $f(x)$ denotes a probability density function chosen as the student's $t$-distribution with degrees of freedom $\nu=10$:
\[f(x) = \frac{\Gamma(\frac{\nu+1}{2})} {\sqrt{\nu\pi}\,\Gamma(\frac{\nu}{2})} \left(1+\frac{x^2}{\nu} \right)^{-\frac{\nu+1}{2}}.\]
The other constants are $\sigma(x)=0.5$ and $x_*=0$.
\item \textbf{Geometric Brownian Motion (GBM):}\\
\\
Next we consider the GBM process with $b(x)=b_0 x$, $\sigma(x)=\sigma_0 x$ and $x_*=1$, where $b_0=0.05$ and $\sigma_0=0.2$.
\item  \textbf{An SDE with a Non-Linear Diffusion Term:}\\
\\
Finally, for this example we take $b(x)=-x$, $\sigma(x)=1/\sqrt{1+x^2}$ and $x_*=0$. 
\end{enumerate}

\subsection{Simulation Settings}
For each example, the ground truth is computed through a particle filter with a discretization level $2^{-8}$ and $10^4$ particles. To compare with the unbiased ST and CS estimators, we run the MLPF algorithm for $L\in\{1,\cdots,7\}$.  For the MLPF algorithm, if $\sigma$ is non-constant, the choice of $N_{0:L}$ is as in Section \ref{sec:cost_disc}.
Otherwise we choose $N_{0:L}$ as specified in \cite[Section 4]{high_freq_ml}.
For the ST and CS estimators, for the case where $\sigma$ is non-constant, we use $P(l)$ as specified in Section \ref{sec:cost_disc}.
If $\sigma$ is constant, \cite[Proposition 4]{rhee} suggests that one can choose $P(l)\propto 2^{-l}\, (l+1)\,(\log_2(l+2))^2$. In practice, we constrain $l$ to the same support
as using the MLPF algorithm where one attempts to achieve a comparable cost. The number of samples $N$ used in the ST and CS estimators was 100 ($\sigma$ constant) or 200
($\sigma$ non-constant). The number of samples $M$ was chosen to have the same $\epsilon$ as the corresponding MLPF ($\epsilon$ is determined by $L$ for the MLPF).
For all PFs/CPFs used, we adopted dynamic resampling (see e.g.~\cite{ddj2012}) according to the effective sample size (ESS) with resampling threshold a quarter of the number of simulated samples (the minimum of the ESS between the levels is used in the case of a CPF).
All results are averaged over 100 runs.

\subsection{Results}

Here we present the plots of $\log$ cost versus $\log$ MSE, where the cost of the MLPF estimator is given by $t\sum_{l=0}^{L} N_l\Delta_l^{-1}$, for the ST estimator the expected cost is given by $tMN\sum_{l=0}^{L} P(l) \Delta_l^{-1}$, and for the CS estimator the expected cost is given by $tMN\sum_{l=0}^{L} P(l) \sum_{q=0}^{l} \Delta_q^{-1}$. The results are presented in Figure \ref{fig:cost_mse}. We remark that in our implementation, for the ST and CS estimators, we did not parallelize over multiple samples. Our results agree with
predicted theory in Section \ref{sec:theory} in that the cost of the unbiased methods is of a similar rate to the MLPF, except with a mild increase.
The CS estimate gives a better cost than the ST estimate and this is consistent with what has been presented about these approaches in the literature.
%

\begin{figure} [H]
\centering
    \subfigure[]{\includegraphics[height=0.34\textwidth]{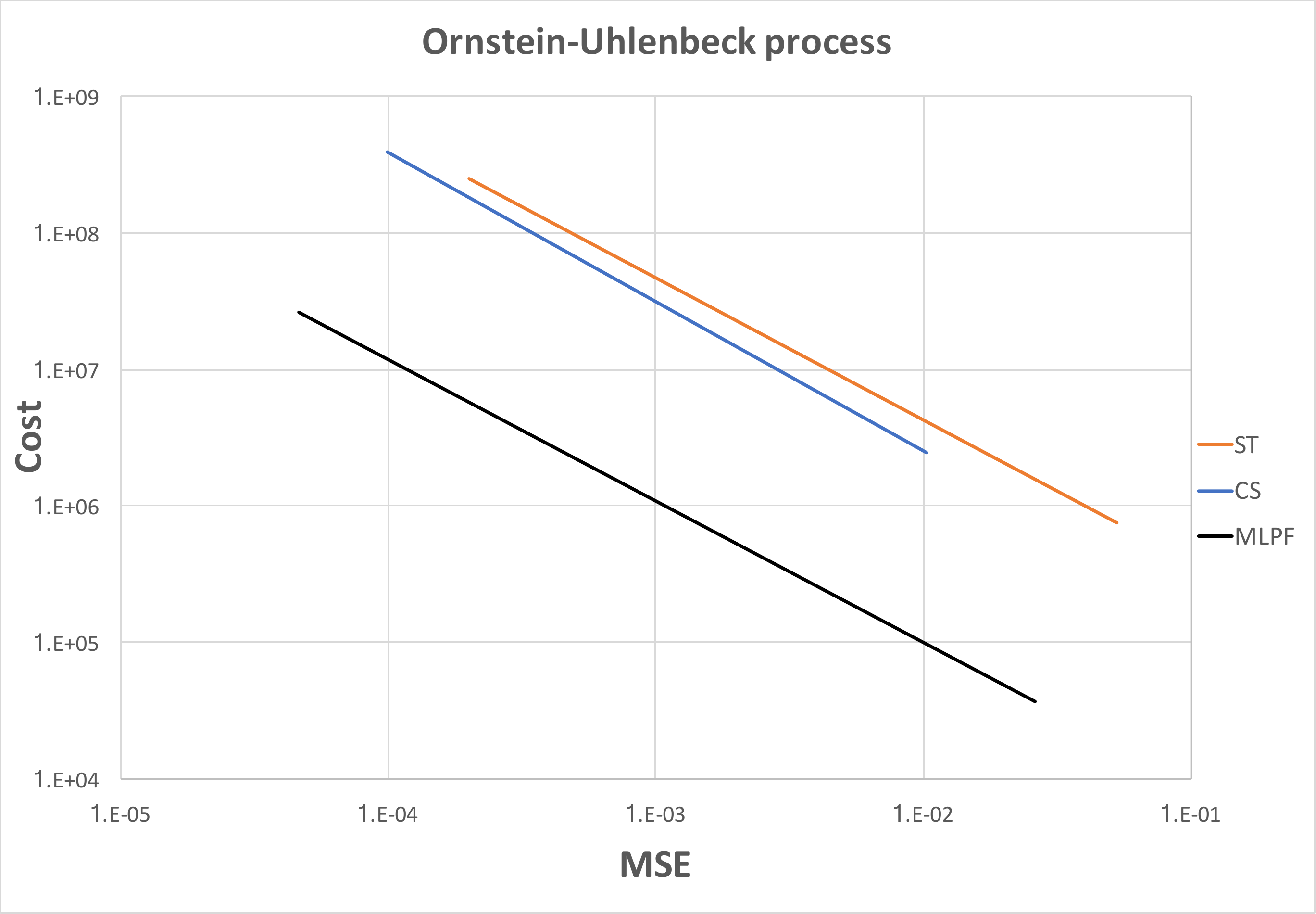} }
    \subfigure[]{\includegraphics[height=0.34\textwidth]{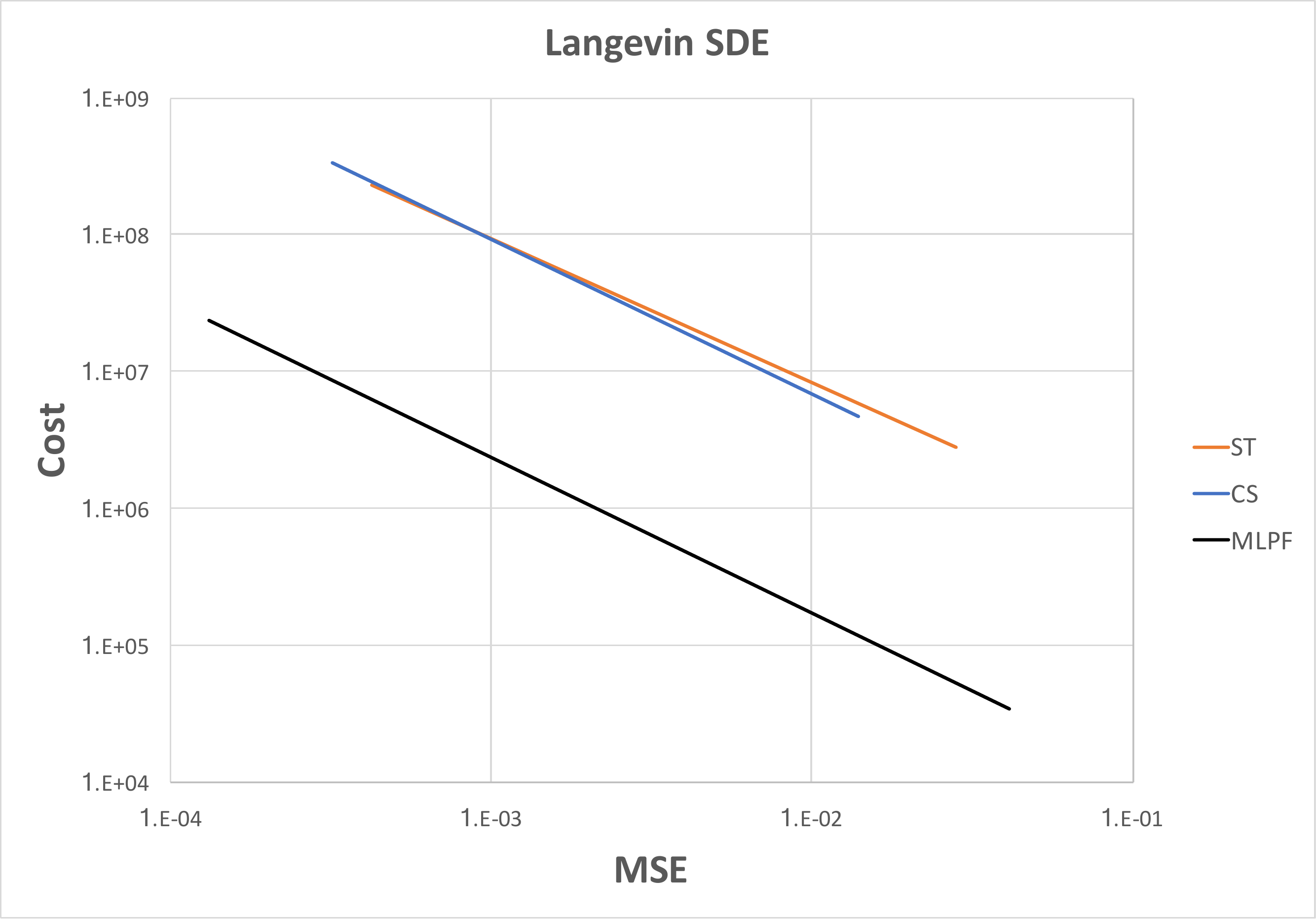}}
     \subfigure[]{\includegraphics[height=0.337\textwidth]{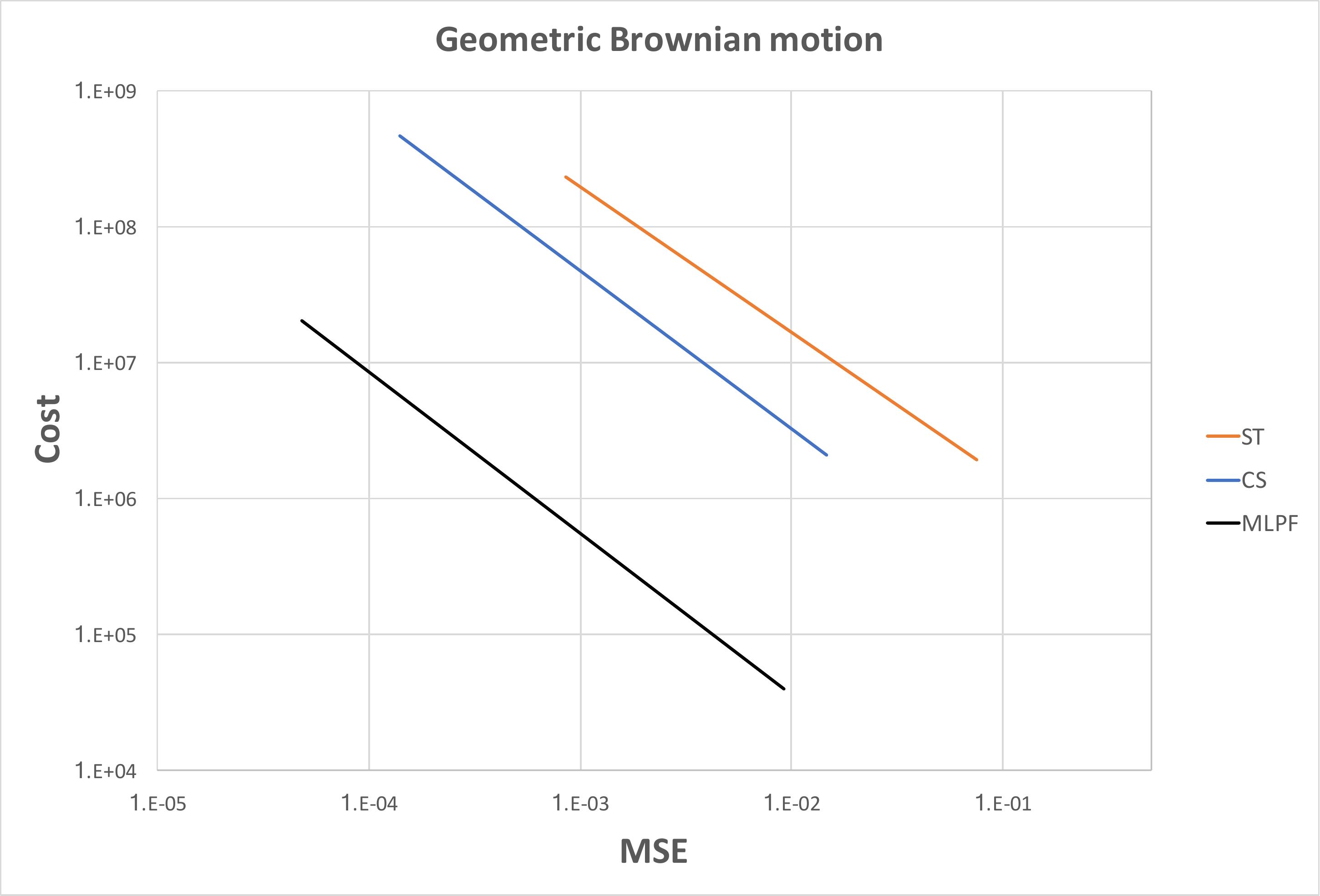}}
   \subfigure[]{\includegraphics[height=0.337\textwidth]{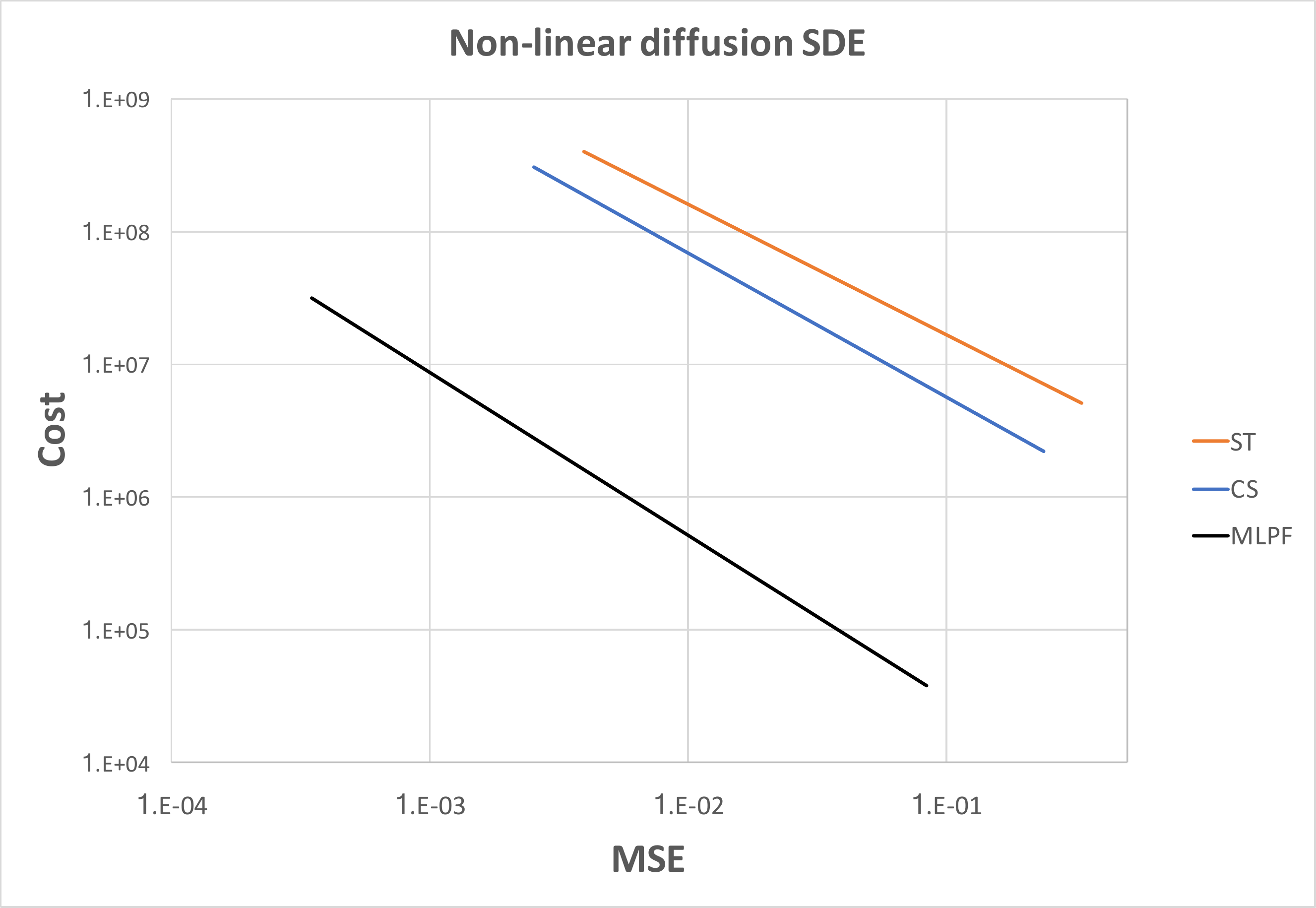}} 
   \caption{Cost versus MSE.} 
    \label{fig:cost_mse}
\end{figure}

\subsubsection*{Acknowledgements}

Both authors were supported by KAUST baseline funding.

\appendix

\section{Proofs}\label{app:proofs}

In order to understand the proofs/results in the main text, this appendix can be read linearly.

Some operators are now defined. Let $(l,p,n)\in\mathbb{N}_0^3$, $n>p$, $(u_p,\varphi)\in E_l\times\mathcal{B}_b(E_l)$
$$
\mathbf{Q}_{p,n}^l(\varphi)(u_p) := \int \varphi(u_n)\Big(\prod_{q=p}^{n-1} \mathbf{G}_q^l(u_q)\Big) \prod_{q=p+1}^{n}M^l(u_{q-1},du_q).
$$
where we use the convention $\mathbf{Q}_{p,p}^l(\varphi)(u_p)=\varphi(u_p)$ and we set $u_p=(x_{p},x_{p+\Delta_l},\dots,x_{p+1})$. For $(p,l)\in\mathbb{N}\times\mathbb{N}$, define the operator $\Phi_p^l:\mathcal{P}(E_l)\rightarrow\mathcal{P}(E_l)$ with $(\mu,\varphi)\in\mathcal{P}(E_l)\times\mathcal{B}_b(E_l)$ as:
\begin{equation}\label{eq:phi_rec}
\Phi_p^l(\mu)(\varphi) :=  \frac{\mu(\mathbf{G}_{p-1}^l\mathbf{M}^l(\varphi))}{\mu(\mathbf{G}_{p-1}^l)}
\end{equation}
where, to clarify, $\mu(\mathbf{G}_{p-1}^l\mathbf{M}^l(\varphi))= \int_{E_l}\mu(d(x_{p-1},x_{p-1+\Delta_l},\dots,x_{p}))\mathbf{G}_{p-1}^l(x_{p-1},x_{p-1+\Delta_l},\dots,x_{p-\Delta_l})M^l(\varphi)(x_p)$.

Now, we write the empirical measure of samples that are generated at level $l$ (resp.~$l-1$) by Algorithm \ref{alg:cpf} at the end of step 1.~or step 2.~for $(t,l,N)\in\mathbb{N}_0\times\mathbb{N}^2$
$$
\pi_t^{l,N}(du) := \frac{1}{N}\sum_{i=1}^N \delta_{\{x_{t:t+1}^{l,i}\}}(du)\quad\textrm{resp.}\quad
\check{\pi}_t^{l-1,N}(du) := \frac{1}{N}\sum_{i=1}^N \delta_{\{\check{x}_{t:t+1}^{l-1,i}\}}(du).
$$
If one just considers a particle filter, as in Algorithm \ref{alg:pf} we use the notation $\pi_t^{l,N}$, $(t,l,N)\in\mathbb{N}_0^2\times\mathbb{N}$ to denote the empirical measure of the samples produced either at the end of step 1.~or step 2.. For $\varphi\in\mathcal{B}_b(\mathbb{R}^{d_x})$, we define, for any $l\in\mathbb{N}_0$, $\pmb{\varphi}^l:E_l\rightarrow\mathbb{R}$
$$
\pmb{\varphi}^l(x_0,x_{\Delta_l},\dots,x_{1}) := \varphi(x_1).
$$
Given the above notation, we have the following martingale (we will define the filtration below) decomposition from \cite[Theorem 7.4.2.]{delm:04} for $(t,l,N,\varphi)\in\mathbb{N}\times\mathbb{N}_0\times\mathbb{N}\times\mathcal{B}_b(\mathbb{R}^{d_x})$:
\begin{equation}\label{eq:mart_nc}
[\gamma_{t,PF}^{l,N}-\gamma_t^l](\varphi) = \sum_{p=0}^{t-1} \gamma_{p,PF}^{l,N}(1)[\pi_p^{l,N}-\Phi_p^l(\pi_{p-1}^{l,N})](\mathbf{Q}_{p,t-1}^l(\mathbf{G}_{t-1}^l\pmb{\varphi}^l))
\end{equation}
where we use the convention $\Phi_0^l(\pi_{-1}^{l,N})(\cdot)=M^l(x_*,\cdot)$. Let $\mathcal{G}_t^l$ be the $\sigma-$algebra generated by the particle filter at level $l\in\mathbb{N}_0$ up-to time $t\in\mathbb{N}_0$ (after step 1.~or step 2.~of Algorithm \ref{alg:pf}, time 0 corresponds to the end of step 1.),
and set $\mathcal{H}_s^l = \mathcal{G}_s^l\otimes\mathcal{Y}_t$ for $s\in\mathbb{N}_0$, with $\mathcal{H}_{-1}^l=\mathcal{Y}_t$ and $t\in\mathbb{N}$ fixed.

In addition, one has for $(t,l,N,\varphi)\in\mathbb{N}^3\times\mathcal{B}_b(\mathbb{R}^{d_x})$
\begin{eqnarray}
 [\gamma_t^l-\gamma_t^{l-1}]^N_{CPF}(\varphi) -  [\gamma_t^l-\gamma_t^{l-1}](\varphi) & = &
 \sum_{p=0}^{t-1} \Big\{\gamma_{p,CPF}^{l,N}(1)[\pi_p^{l,N}-\Phi_p^l(\pi_{p-1}^{l,N})](\mathbf{Q}_{p,t-1}^l(\mathbf{G}_{t-1}^l\pmb{\varphi}^l)) - \nonumber\\& &
 \check{\gamma}_{p,CPF}^{l-1,N}(1)[\check{\pi}_p^{l-1,N}-\Phi_p^{l-1}(\check{\pi}_{p-1}^{l-1,N})](\mathbf{Q}_{p,t-1}^{l-1}(\mathbf{G}_{t-1}^{l-1}\pmb{\varphi}^{l-1}))\Big\}
 \label{eq:mart_nc_cpf}
 \end{eqnarray}
 where we use the convention $\Phi_0^{l-1}(\check{\pi}_{-1}^{l-1,N}(\cdot)=M^{l-1}(x_*,\cdot)$ and we use the notation
 \begin{eqnarray*}
 \gamma_{p,CPF}^{l,N}(1) & = & \prod_{q=0}^{p-1}\pi_q^{l,N}(\mathbf{G}_q^l) \\
  \check{\gamma}_{p,CPF}^{l-1,N}(1) & = & \prod_{q=0}^{p-1}\check{\pi}_q^{l-1,N}(\mathbf{G}_q^{l-1}).
 \end{eqnarray*}
 Let $\mathcal{\check{G}}_t^l$ be the $\sigma-$algebra generated by the coupled particle filter at level $l\in\mathbb{N}$ up-to time $t\in\mathbb{N}_0$ (after step 1.~or step 2.~of Algorithm \ref{alg:cpf}, time 0 corresponds to the end of step 1.),
and set $\mathcal{\check{H}}_s^l = \mathcal{\check{G}}_s^l\otimes\mathcal{Y}_t$ for $s\in\mathbb{N}_0$, with $\mathcal{\check{H}}_{-1}^l=\mathcal{Y}_t$ and $t\in\mathbb{N}$ fixed.

\begin{proof}[Proof of Proposition \ref{prop:ub_nc}]
We have, almost surely, that for any $(t,l,N,\varphi)\in\mathbb{N}\times\mathbb{N}_0\times\mathbb{N}\times\mathcal{B}_b(\mathbb{R}^{d_x})$ and $s\in\{-1,\dots,t-2\}$
$$
\mathbb{\overline{E}}[[\gamma_{t,PF}^{l,N}-\gamma_t^l](\varphi)|\mathcal{H}_{s}^l] = \sum_{p=0}^{s} \gamma_{p,PF}^{l,N}(1)[\pi_p^{l,N}-\Phi_p^l(\pi_{p-1}^{l,N})](\mathbf{Q}_{p,t-1}^l(\mathbf{G}_{t-1}^l\pmb{\varphi}^l))
$$
and hence that 
$$
\mathbb{\overline{E}}[[\gamma_{t,PF}^{l,N}-\gamma_t^l](\varphi)|\mathcal{H}_{-1}^l] = \mathbb{\overline{E}}[[\gamma_{t,PF}^{l,N}-\gamma_t^l](\varphi)|\mathcal{Y}_{t}] = 0.
$$
In an almost identical argument, for any $(t,l,N,\varphi)\in\mathbb{N}\times\mathbb{N}\times\mathbb{N}\times\mathcal{B}_b(\mathbb{R}^{d_x})$, almost surely
$$
\mathbb{\overline{E}}[ [\gamma_t^l-\gamma_t^{l-1}]^N_{CPF}(\varphi) -  [\gamma_t^l-\gamma_t^{l-1}](\varphi)|\mathcal{\check{H}}_{-1}^l] = 
\mathbb{\overline{E}}[ [\gamma_t^l-\gamma_t^{l-1}]^N_{CPF}(\varphi) -  [\gamma_t^l-\gamma_t^{l-1}](\varphi)|\mathcal{Y}_t] = 0
$$
which allows one to conclude the result.
\end{proof}

\begin{prop}\label{prop:cpf_bound}
Assume (D\ref{hyp_diff:1}). Then for any $(t,q)\in\mathbb{N}\times\mathbb{N}$, there exists a $C<+\infty$ such that
for any $(l,N,\varphi)\in\mathbb{N}^2\times \mathcal{B}_b(\mathbb{R}^{d_x})\cap\textrm{\emph{Lip}}_{\|\cdot\|_2}(\mathbb{R}^{d_x})$
$$
\mathbb{\overline{E}}[|
[\gamma_t^l-\gamma_t^{l-1}]^N_{CPF}(\varphi) -  [\gamma_t^l-\gamma_t^{l-1}](\varphi)|^q]^{1/q} \leq C(\|\varphi\|+\|\varphi\|_{\textrm{\emph{Lip}}})\frac{\Delta_l^{1/4}}{\sqrt{N}}.
$$
\end{prop}

\begin{proof}
Throughout $C$ is a finite constant whose value may change on appearance and does not depend upon $l$ nor $N$.
Our proof is by strong induction on $t$. Consider the case $t=1$, then using \eqref{eq:mart_nc_cpf}
\begin{eqnarray*}
\mathbb{\overline{E}}[|
[\gamma_1^l-\gamma_1^{l-1}]^N_{CPF}(\varphi) -  [\gamma_1^l-\gamma_1^{l-1}](\varphi)|^q]^{1/q} & = & 
\mathbb{\overline{E}}[|
\pi_0^{l,N}(\mathbf{G}_{0}^l\pmb{\varphi}^l)-M^l(\mathbf{G}_{0}^l\pmb{\varphi}^l)(x_*) - 
\check{\pi}_0^{l-1,N}(\mathbf{G}_{0}^{l-1}\pmb{\varphi}^{l-1})+ \\ & &M^{l-1}(\mathbf{G}_{0}^{l-1}\pmb{\varphi}^{l-1})(x_*)
|^q]^{1/q}.
\end{eqnarray*}
Applying the Marcinkiewicz-Zygmund and Jensen inequalites, one can deduce that
$$
\mathbb{\overline{E}}[|
[\gamma_1^l-\gamma_1^{l-1}]^N_{CPF}(\varphi) -  [\gamma_1^l-\gamma_1^{l-1}](\varphi)|^q]^{1/q} \leq
C\frac{1}{\sqrt{N}}\mathbb{\overline{E}}[|
\mathbf{G}_{0}^l(U_0^{l,i})\varphi(X_1^{l,i}) - 
\mathbf{G}_{0}^l(\check{U}_0^{l-1,i})\varphi(\check{X}_1^{l-1,i})
|^q]^{1/q}.
$$
By \cite[Lemma A.8.]{high_freq_ml} one can deduce that
$$
\mathbb{\overline{E}}[|
[\gamma_1^l-\gamma_1^{l-1}]^N_{CPF}(\varphi) -  [\gamma_1^l-\gamma_1^{l-1}](\varphi)|^q]^{1/q} \leq
C(\|\varphi\|+\|\varphi\|_{\textrm{Lip}})\frac{\Delta_l^{1/2}}{\sqrt{N}}
$$
and hence the initialization follows.

We now assume the result at ranks $1,\dots,t-1$ and consider $t$. We have, almost surely, that (via \eqref{eq:mart_nc_cpf})
\begin{equation}\label{eq:main_proof4}
 [\gamma_t^l-\gamma_t^{l-1}]^N_{CPF}(\varphi) -  [\gamma_t^l-\gamma_t^{l-1}](\varphi) = \sum_{j=1}^3 T_j
\end{equation}
where
\begin{eqnarray*}
T_1 & = &  \sum_{p=0}^{t-1} \Big[[\gamma_{p}^{l} - \gamma_{p}^{l-1}]_{CPF}^{N}(1) - 
[\gamma_{p}^{l} - \gamma_{p}^{l-1}](1)
\Big]
[\pi_p^{l,N}-\Phi_p^l(\pi_{p-1}^{l,N})](\mathbf{Q}_{p,t-1}^l(\mathbf{G}_{t-1}^l\pmb{\varphi}^l)) \\
T_2 & = & \sum_{p=0}^{t-1}[\gamma_{p}^{l} - \gamma_{p}^{l-1}](1)
[\pi_p^{l,N}-\Phi_p^l(\pi_{p-1}^{l,N})](\mathbf{Q}_{p,t-1}^l(\mathbf{G}_{t-1}^l\pmb{\varphi}^l)) \\
T_3 & = &  \sum_{p=0}^{t-1}  \check{\gamma}_{p,CPF}^{l-1,N}(1)\Big[
[\pi_p^{l,N}-\Phi_p^l(\pi_{p-1}^{l,N})](\mathbf{Q}_{p,t-1}^l(\mathbf{G}_{t-1}^l\pmb{\varphi}^l)) -
[\check{\pi}_p^{l-1,N}-\Phi_p^{l-1}(\check{\pi}_{p-1}^{l-1,N})](\mathbf{Q}_{p,t-1}^{l-1}(\mathbf{G}_{t-1}^{l-1}\pmb{\varphi}^{l-1}))
\Big].
\end{eqnarray*}
By using Minkowski's inequality, we can upper-bound the $\mathbb{L}_q-$norms of $T_1-T_3$ independently. For $T_1$, again applying the Minkowski
inequality $t$ times, one has
$$
\mathbb{\overline{E}}[|T_1|^q]^{1/q} \leq  \sum_{p=0}^{t-1} \mathbb{\overline{E}}\Big[\Big|
\Big[[\gamma_{p}^{l} - \gamma_{p}^{l-1}]_{CPF}^{N}(1) - 
[\gamma_{p}^{l} - \gamma_{p}^{l-1}](1)
\Big]
[\pi_p^{l,N}-\Phi_p^l(\pi_{p-1}^{l,N})](\mathbf{Q}_{p,t-1}^l(\mathbf{G}_{t-1}^l\pmb{\varphi}^l))\Big|^q\Big]^{1/q}.
$$
Applying Cauchy-Schwarz and the induction hypothesis, along with \cite[Lemma A.10.]{high_freq_ml} yields
\begin{equation}\label{eq:main_proof1}
\mathbb{\overline{E}}[|T_1|^q]^{1/q} \leq C(\|\varphi\|+\|\varphi\|_{\textrm{Lip}})\frac{\Delta_l^{1/4}}{\sqrt{N}}.
\end{equation}
For $T_2$, applying the Minkowski
inequality $t$ times and the Cauchy-Schwarz inequality
$$
\mathbb{\overline{E}}[|T_2|^q]^{1/q} \leq  \sum_{p=0}^{t-1} \Big\{\mathbb{\overline{E}}[|[\gamma_{p}^{l} - \gamma_{p}^{l-1}](1)|^{2q}]^{1/(2q)}
\mathbb{\overline{E}}[|[\pi_p^{l,N}-\Phi_p^l(\pi_{p-1}^{l,N})](\mathbf{Q}_{p,t-1}^l(\mathbf{G}_{t-1}^l\pmb{\varphi}^l))|^{2q}]^{1/(2q)}
\Big\}.
$$
For the left expectation, one can apply Lemma \ref{lem:disc_zakai} 1.~and for the right the (conditional) Marcinkiewicz-Zygmund and Jensen inequalies
along with \cite[Lemma A.10.]{high_freq_ml}, to give
\begin{equation}\label{eq:main_proof2}
\mathbb{\overline{E}}[|T_2|^q]^{1/q} \leq C(\|\varphi\|+\|\varphi\|_{\textrm{Lip}})\frac{\Delta_l^{1/2}}{\sqrt{N}}.
\end{equation}
For $T_3$, using a similar strategy as for $T_1$ and $T_2$ one has the upper-bound
$$
\mathbb{\overline{E}}[|T_3|^q]^{1/q} \leq
$$
$$
\sum_{p=0}^{t-1}\mathbb{\overline{E}}[\check{\gamma}_{p,CPF}^{l-1,N}(1)^{2q}]^{1/(2q)}
\mathbb{\overline{E}}\Big[
\Big|
[\pi_p^{l,N}-\Phi_p^l(\pi_{p-1}^{l,N})](\mathbf{Q}_{p,t-1}^l(\mathbf{G}_{t-1}^l\pmb{\varphi}^l)) -
[\check{\pi}_p^{l-1,N}-\Phi_p^{l-1}(\check{\pi}_{p-1}^{l-1,N})](\mathbf{Q}_{p,t-1}^{l-1}(\mathbf{G}_{t-1}^{l-1}\pmb{\varphi}^{l-1}))
\Big|^{2q}\Big]^{1/(2q)}.
$$
For the left expectation, one can using the bound \cite[(14)]{high_freq_ml} and then take expectations w.r.t.~the data to yield that
$$
\mathbb{\overline{E}}[\check{\gamma}_{p,CPF}^{l-1,N}(1)^{2q}]^{1/(2q)} \leq C
$$
where $C$ does not depend upon $l$. For the right expectation one can use the (conditional) Marcinkiewicz-Zygmund and Jensen inequalites, one can deduce that
$$
\mathbb{\overline{E}}[|T_3|^q]^{1/q} \leq \frac{C}{\sqrt{N}}\sum_{p=0}^{t-1}
\mathbb{\overline{E}}\Big[\Big|
\mathbf{Q}_{p,t-1}^l(\mathbf{G}_{t-1}^l\pmb{\varphi}^l)(U_p^{l,1}) -
\mathbf{Q}_{p,t-1}^{l-1}(\mathbf{G}_{t-1}^{l-1}\pmb{\varphi}^{l-1})(\check{U}_p^{l-1,1})
\Big|^{2q}\Big]^{1/(2q)}.
$$
The expectation in the summand can be controlled by using a very similar approach to the proof of \cite[Lemma A.4.]{high_freq_ml}, to yield
\begin{equation}\label{eq:main_proof3}
\mathbb{\overline{E}}[|T_3|^q]^{1/q} \leq C(\|\varphi\|+\|\varphi\|_{\textrm{Lip}})\frac{\Delta_l^{1/4}}{\sqrt{N}}.
\end{equation}
Noting \eqref{eq:main_proof4} along with \eqref{eq:main_proof1}-\eqref{eq:main_proof3}, the proof can be easily concluded.
\end{proof}

\begin{rem}\label{rem:pf_bound}
It straight-forward to deduce that using the representation \eqref{eq:mart_nc} and the strategy used in the proof above, that one can prove the following under
(D\ref{hyp_diff:1}). For any $(t,q)\in\mathbb{N}\times\mathbb{N}$, there exists a $C<+\infty$ such that for any $(l,\varphi)\in\mathbb{N}_0\times \mathcal{B}_b(\mathbb{R}^{d_x})\cap\textrm{\emph{Lip}}_{\|\cdot\|_2}(\mathbb{R}^{d_x})$:
$$
\mathbb{\overline{E}}[|[\gamma_{t,PF}^{l,N}-\gamma_t^l](\varphi)|^q]^{1/q} \leq C(\|\varphi\|+\|\varphi\|_{\textrm{\emph{Lip}}})\frac{1}{\sqrt{N}}.
$$
\end{rem}

\begin{lem}\label{lem:disc_zakai}
Assume (D\ref{hyp_diff:1}). Then for any $(t,q)\in\mathbb{N}\times\mathbb{N}$, there exists a $C<+\infty$ such that:
\begin{enumerate}
\item{for any $(l,\varphi)\in\mathbb{N}\times \mathcal{B}_b(\mathbb{R}^{d_x})\cap\textrm{\emph{Lip}}_{\|\cdot\|_2}(\mathbb{R}^{d_x})$
$$
\mathbb{\overline{E}}[|[\gamma_t^l-\gamma_t^{l-1}](\varphi)|^q]^{1/q} \leq C(\|\varphi\|+\|\varphi\|_{\textrm{\emph{Lip}}})\Delta_l^{1/2}
$$
}
\item{for any $(l,\varphi)\in\mathbb{N}_0\times \mathcal{B}_b(\mathbb{R}^{d_x})\cap\textrm{\emph{Lip}}_{\|\cdot\|_2}(\mathbb{R}^{d_x})$
$$
\mathbb{\overline{E}}[|[\gamma_t^l-\gamma_t](\varphi)|^q]^{1/q} \leq C(\|\varphi\|+\|\varphi\|_{\textrm{\emph{Lip}}})\Delta_l^{1/2}.
$$
}
\end{enumerate}
\end{lem}

\begin{proof}
The first result is \cite[Lemma A.8.]{high_freq_ml} and the second is \cite[Lemma A.5.]{high_freq_ml}.
\end{proof}

\end{document}